\documentclass[english]{article}
\usepackage[T1]{fontenc}
\usepackage[latin9]{inputenc}
\usepackage{geometry}
\usepackage{amsmath}
\usepackage{amsthm}
\usepackage{amssymb}
\usepackage{wasysym}

\makeatletter
\numberwithin{equation}{section}
\numberwithin{figure}{section}
\theoremstyle{plain}
\newtheorem{thm}{\protect\theoremname}
\theoremstyle{plain}
\newtheorem{conjecture}[thm]{\protect\conjecturename}
\theoremstyle{definition}
\newtheorem{defn}[thm]{\protect\definitionname}
\theoremstyle{plain}
\newtheorem{cor}[thm]{\protect\corollaryname}
\theoremstyle{remark}
\newtheorem*{rem*}{\protect\remarkname}
\theoremstyle{plain}
\newtheorem{lem}[thm]{\protect\lemmaname}
\theoremstyle{remark}
\newtheorem{claim}[thm]{\protect\claimname}
\theoremstyle{remark}
\newtheorem*{acknowledgement*}{\protect\acknowledgementname}

\usepackage{hyperref}
\hypersetup{
    colorlinks,
    allcolors=red
}
\usepackage[lined,boxed,ruled,norelsize,algo2e]{algorithm2e}
\@addtoreset{section}{part}
\usepackage{tikz}

\makeatother

\usepackage{babel}
\providecommand{\acknowledgementname}{Acknowledgement}
\providecommand{\claimname}{Claim}
\providecommand{\conjecturename}{Conjecture}
\providecommand{\corollaryname}{Corollary}
\providecommand{\definitionname}{Definition}
\providecommand{\lemmaname}{Lemma}
\providecommand{\remarkname}{Remark}
\providecommand{\theoremname}{Theorem}

\begin{document}
\global\long\def\defeq{\stackrel{\mathrm{{\scriptscriptstyle def}}}{=}}%
\global\long\def\norm#1{\left\Vert #1\right\Vert }%
\global\long\def\R{\mathbb{R}}%
 
\global\long\def\Rn{\mathbb{R}^{n}}%
\global\long\def\tr{\mathrm{Tr}}%
\global\long\def\diag{\mathrm{diag}}%
\global\long\def\cov{\mathrm{Cov}}%
\global\long\def\E{\mathbb{E}}%
\global\long\def\P{\mathbb{P}}%
\global\long\def\Var{\mathrm{Var}}%
\global\long\def\rank{\mathrm{rank}}%
\global\long\def\lref#1{\text{Lem }\ltexref{#1}}%
\global\long\def\lreff#1#2{\text{Lem }\ltexref{#1}.\ltexref{#1#2}}%
\global\long\def\ltexref#1{\ref{lem:#1}}
\global\long\def\ttag#1{\tag{#1}}%
\global\long\def\cirt#1{\raisebox{.5pt}{\textcircled{\raisebox{-.9pt}{#1}}}}%
\global\long\def\Ent{\mathrm{Ent}}%
\global\long\def\vol{\mathrm{vol}}%
\global\long\def\spe{\mathrm{op}}%
\global\long\def\op{\mathrm{op}}%

\setcounter{page}{0}
\title{Eldan's Stochastic Localization and the KLS Conjecture: Isoperimetry,
Concentration and Mixing\thanks{This update unifies the results from \cite{LeeV17KLS} and \cite{lee2018stochastic}.}}
\author{Yin Tat Lee\thanks{University of Washington and Microsoft Research, yintat@uw.edu},
Santosh S. Vempala\thanks{Georgia Tech, vempala@gatech.edu}}
\maketitle
\begin{abstract}
We show that the Cheeger constant for $n$-dimensional isotropic logconcave
measures is $O(n^{1/4})$, improving on the previous best bound of
$O(n^{1/3}\sqrt{\log n})$$.$ As corollaries we obtain the same improved
bound on the thin-shell estimate, Poincar� constant and Lipschitz
concentration constant and an alternative proof of this bound for
the isotropic (slicing) constant; it also follows that the ball walk
for sampling from an isotropic logconcave density in $\R^{n}$ converges
in $O^{*}(n^{2.5})$ steps from a warm start. The proof is based on
gradually transforming any logconcave density to one that has a significant
Gaussian factor via a Martingale process.

Extending this proof technique, we prove that the log-Sobolev constant
of any isotropic logconcave density in $\R^{n}$ with support of diameter
$D$ is $\Omega(1/D)$, resolving a question posed by Frieze and Kannan
in 1997. This is asymptotically the best possible estimate and improves
on the previous bound of $\Omega(1/D^{2})$ by Kannan-Lov�sz-Montenegro.
It follows that for any isotropic logconcave density, the ball walk
with step size $\delta=\Theta(1/\sqrt{n})$ mixes in $O\left(n^{2}D\right)$
proper steps from \emph{any }starting point. This improves on the
previous best bound of $O(n^{2}D^{2})$ and is also asymptotically
tight. 

The new bound leads to the following large deviation inequality for
an $L$-Lipschitz function $g$ over an isotropic logconcave density
$p$: for any $t>0$, 
\[
\P_{x\sim p}\left(\left|g(x)-\bar{g}\right|\geq L\cdot t\right)\leq\exp(-\frac{c\cdot t^{2}}{t+\sqrt{n}})
\]
where $\bar{g}$ is the median or mean of $g$ for $x\sim p$; this
generalizes and improves on previous bounds by Paouris and by Guedon-Milman.
The technique also bounds the ``small ball'' probability in terms
of the Cheeger constant, and recovers the current best bound. 
\end{abstract}

\section{Introduction}

In this paper we study the Cheeger constant and the log-Sobolev constant
of logconcave measures in $n$-dimensional Euclidean space. These
fundamental parameters, which we will define presently, have many
important connections and applications (cf. \cite{Ledoux1999,brazitikos2014geometry}). 

The isoperimetry of a subset is the ratio of the measure of the boundary
of the subset to the measure of the subset or its complement, whichever
is smaller. The minimum such ratio over all subsets is the Cheeger
constant, also called expansion or isoperimetric coefficient. This
fundamental constant appears in many settings, e.g., graphs and convex
bodies, and plays an essential role in many lines of study.

\subsection{Cheeger constant}

In the geometric setting, the KLS hyperplane conjecture \cite{KLS95}
asserts that for any distribution with a logconcave density, the minimum
expansion is approximated by that of a halfspace, up to a universal
constant factor. Thus, if the conjecture is true, the Cheeger constant
can be essentially determined simply by examining hyperplane cuts.
More precisely, here is the statement. We use $c,C$ for absolute
constants, $a\lesssim b$ to denote $a\leq c\cdot b$ for some absolute
constant $c$, and $\norm A_{\op}$ for the spectral/operator norm
of a matrix $A$.
\begin{conjecture}[\cite{KLS95}]
 For any logconcave density p in $\R^{n}$ with covariance matrix
$A$, 
\[
\frac{1}{\psi_{p}}\defeq\inf_{S\subseteq\R^{n}}\frac{\int_{\partial S}p(x)dx}{\min\left\{ \int_{S}p(s)dx,\int_{\R^{n}\setminus S}p(x)dx\right\} }\gtrsim\frac{1}{\sqrt{\|A\|_{\op}}}.
\]
\end{conjecture}

For an isotropic logconcave density (all eigenvalues of its covariance
matrix are equal to $1$), the conjectured isoperimetric ratio is
an absolute constant. Note that the KLS constant $\psi_{p}$ is the
reciprocal of the Cheeger constant (this will be more convenient for
comparisons with other constants). The conjecture was formulated by
Kannan, Lovász and Simonovits in the course of their study of the
convergence of a random process (the ball walk) in a convex body.
They proved the following weaker bound. 
\begin{thm}[\cite{KLS95}]
\label{thm:KLS_thm} For any logconcave density $p$ in $\R^{n}$
with covariance matrix $A$, the KLS constant satisfies
\[
\psi_{p}\lesssim\sqrt{\tr(A)}.
\]
\end{thm}

For an isotropic distribution, the theorem gives a bound of $O\left(\sqrt{n}\right)$,
while the conjecture says $O\left(1\right)$. The conjecture has several
important consequences. For example, it implies that the ball walk
mixes in $O^{*}\left(n^{2}\right)$ steps from a warm start in any
isotropic convex body (or logconcave density) in $\R^{n}$; this is
the best possible bound, and is tight e.g., for a hypercube. The KLS
conjecture has become central to modern asymptotic convex geometry.
It is equivalent to a bound on the spectral gap of isotropic logconcave
functions \cite{Ledoux04}. Although it was formulated due to an algorithmic
motivation, it implies several well-known conjectures in asymptotic
convex geometry. We describe these next. 

The \emph{thin-shell} conjecture (also known as the \emph{variance
hypothesis)} \cite{Anttila2003,Bobkov2003} says the following. 
\begin{conjecture}[Thin-shell]
 For a random point $X$ from an isotropic logconcave density $p$
in $\R^{n}$, 
\[
\sigma_{p}^{2}\defeq\E((\|X\|-\sqrt{n})^{2})\lesssim1.
\]
 
\end{conjecture}

It implies that a random point $X$ from an isotropic logconcave density
lies in a constant-width annulus (a thin shell) with constant probability.
Noting that 
\[
\sigma_{p}^{2}=\E((\|X\|-\sqrt{n})^{2})\le\frac{1}{n}\Var(\|X\|^{2})\lesssim\sigma_{p}^{2},
\]
the conjecture is equivalent to asserting that $\Var(\norm X^{2})\lesssim n$
for an isotropic logconcave density. The following connection is well-known:
$\sigma_{p}\lesssim\psi_{p}$. The current best bound is $\sigma_{p}\lesssim n^{\frac{1}{3}}$
by Guedon and Milman \cite{GuedonM11}, improving on a line of work
that started with Klartag \cite{Klartag2007,Klartag2007b,Fleury2010}.
Eldan \cite{Eldan2013} has shown that the reverse inequality holds
approximately, in a worst-case sense, namely the worst possible KLS
constant over all isotropic logconcave densities in $\R^{n}$ is bounded
by the thin-shell estimate to within roughly a logarithmic factor
in the dimension. This yields the current best bound of $\psi_{p}\lesssim n^{\frac{1}{3}}\sqrt{\log n}$.
A weaker inequality was shown earlier by Bobkov \cite{Bobkov2007}
(see also \cite{Milman2009}). 

The \emph{slicing} conjecture, also called the\emph{ hyperplane conjecture}
\cite{Bourgain1986,Ball88} is the following. 
\begin{conjecture}[Slicing/Isotropic constant]
 Any convex body of unit volume in $\R^{n}$ contains a hyperplane
section of at least constant volume. Equivalently, for any convex
body $K$ of unit volume with covariance matrix $L_{K}^{2}I$, the
isotropic constant $L_{K}\lesssim1$. 
\end{conjecture}

The isotropic constant of a general isotropic logconcave density $p$
is defined as $L_{p}=p(0)^{1/n}$. The best current bound is $L_{p}\lesssim n^{1/4}$,
due to Klartag \cite{Klartag2006}, improving on Bourgain's bound
of $L_{p}\lesssim n^{1/4}\log n$ \cite{Bourgain91}. The study of
this conjecture has played an influential role in the development
of convex geometry over the past several decades. It was shown by
Ball that the KLS conjecture implies the slicing conjecture. More
recently, Eldan and Klartag \cite{EldanK2011} showed that the thin
shell conjecture implies slicing, and therefore an alternative (and
stronger) proof that KLS implies slicing: $L_{p}\lesssim\sigma_{p}\lesssim\psi_{p}$.

Here are a few applications of the KLS bound.
\begin{thm}[Poincaré constant \cite{Mazja60,Cheeger69}]
\label{thm:Poincare} For any isotropic logconcave density p in $\R^{n}$
and any smooth function $g$, we have
\[
\Var_{x\sim p}g(x)\lesssim\psi_{p}^{2}\cdot\E_{x\sim p}\norm{\nabla g(x)}_{2}^{2}.
\]
\end{thm}

The Cheeger and log-Sobolev constants also play an important role
in the phenomenon known as concentration of measure. The following
result is due to Gromov and Milman. 
\begin{thm}[Lipschitz concentration \cite{GromovM83}]
 For any $L$-Lipschitz function $g$ in $\R^{n},$ and isotropic
logconcave density $p$, 
\[
\P_{x\sim p}\left(\left|g(x)-\E g\right|\ge\psi_{p}\cdot L\cdot t\right)\le e^{-c\cdot t}.
\]
\end{thm}

A different bound, independent of the Cheeger constant, for the deviation
in length of a random vector was given in a celebrated paper by Paouris
\cite{Paouris2006} and improved by Guedon and Milman \cite{GuedonM11}
(Paouris' result has only the second term in the minimum below, and
is sharp when $t\gtrsim\sqrt{n}$). 
\begin{thm}[\cite{GuedonM11,Paouris2006}]
 For any isotropic logconcave density $p$,
\[
\P_{x\sim p}\left(\left|\norm x-\sqrt{n}\right|\ge t\right)\le e^{-c\cdot\min\left\{ \frac{t^{3}}{n},t\right\} }.
\]
\end{thm}

Our tight log-Sobolev bound will be useful in proving an improved
concentration inequality. Finally, the Cheeger constant is useful
for proving a central limit theorem for convex sets.
\begin{thm}[Central Limit Theorem \cite{Anttila2003}]
Let $K$ be an isotropic symmetric convex set. Let $g_{\theta}(s)=\text{vol}(K\cap\{x^{T}\theta=s\})$
and $g(s)=\frac{1}{\sqrt{2\pi}}\exp(-\frac{s^{2}}{2})$. There are
universal constants $c_{1}$, $c_{2}>0$ such that for any $\delta>0$,
we have 
\[
\text{vol}\left(\left\{ \theta\in S^{n-1}:\text{ }\left|\int_{-t}^{t}g_{\theta}(s)ds-\int_{-t}^{t}g(s)ds\right|\leq c_{1}(\delta+\frac{\psi_{K}}{\sqrt{n}})\text{ for every }t\in\R\right\} \right)\geq1-ne^{-c_{2}\delta^{2}n}.
\]
\end{thm}

For more background on these conjectures, we refer the reader to \cite{brazitikos2014geometry,alonso2015approaching,ArtsteinGM2015}. 

\subsection{Log-Sobolev constant}

The KLS conjecture was motivated by the study of the convergence of
a Markov chain, the \emph{ball walk} in a convex body. To sample uniformly
from a convex body, the ball walk starts at some point in the body,
picks a random point in the ball of radius $\delta$ around the current
point and if the chosen point is in the body, it steps to the new
point. It can be generalized to sampling any logconcave density by
using a Metropolis filter. As shown in \cite{KLS97}, the ball walk
applied to a logconcave density mixes in $O^{*}(n^{2}\psi_{p}^{2})$
steps from a warm start, which using the current-best bound \cite{LeeV17KLS}
is $O^{*}(n^{2.5})$. Looking closer, from a starting distribution
$Q_{o}$, the distance of the distribution obtained after $t$ steps
from $Q_{o}$ to the stationary distribution $Q$ drops as 
\[
d(Q_{t},Q)\le d(Q_{0},Q)\left(1-\frac{\phi^{2}}{2}\right)^{t}
\]
where $\phi$ is the conductance of the Markov chain and $d(.,.)$
is the $\chi$-squared distance. The conductance can be viewed as
the Cheeger constant of the Markov chain. Thus the number of steps
needed is $O(\phi^{-2}\log(1/d(Q_{0},Q)))$. Roughly speaking, for
the ball walk applied to a logconcave density $p$, the conductance
is $\Omega(1/(n\psi_{p}))$, leading to the bound of $O^{*}(n^{2}\psi_{p}^{2})$
steps from a warm start. The dependence on the starting distribution
leads to an additional factor of $n$ in the running time when the
starting distribution is not warm (i.e., $d(Q_{1},Q)$ after one step
can be $e^{-\Omega(n)}$ ). This is a general issue for Markov chains.
One way to address this is via the log-Sobolev constant \cite{diaconis1996,Ledoux1999}.
We first define it for a density, then for a Markov chain.
\begin{defn}
For a density $p$, the log-Sobolev constant $\rho_{p}$ is the largest
$\rho$ such that for every smooth function $f:\R^{n}\rightarrow\R$
with $\int f^{2}dp=1$, we have
\[
\frac{\rho}{2}\int f^{2}\log f^{2}dp\leq\int\norm{\nabla f}^{2}dp.
\]
\end{defn}

A closely related parameter is the following.
\begin{defn}
The \emph{log-Cheeger} constant $\kappa_{p}$ of a density $p$ in
$\R^{n}$ is
\end{defn}

\[
\kappa_{p}=\inf_{S\subseteq\R^{n}}\frac{p(\partial S)}{\min\left\{ p(S),p(\R^{n}\setminus S)\right\} \sqrt{\log\left(\frac{1}{p(S)}\right)}}.
\]
It is known that $\rho_{p}=\Theta(\kappa_{p}^{2})$ (see e.g., \cite{ledoux1994simple}).
The log-Cheeger constant shows more explicitly that the log-Sobolev
constant is a uniform bound on the expansion ``at every scale''. 

For a reversible Markov chain with transition operator $P$ and stationary
density $Q$, we can define the log-Sobolev constant $\rho(P)$ be
the largest $\rho$ such that for every smooth functions satisfying
$f:\R^{n}\rightarrow\R$ with $\int f^{2}dp=1$, we have
\[
\frac{\rho}{2}\int f^{2}\log f^{2}dQ\leq\int_{x}\int_{y}\norm{f(x)-f(y)}^{2}P(x,y)dQ(x).
\]
Diaconis and Saloff-Coste \cite{diaconis1996} show that the distribution
after $t$ steps satisfies $\Ent(Q_{t})\le e^{-c\cdot\rho(P)t}\cdot\Ent(Q_{0})$
where $\Ent(Q_{t})=\int Q_{t}\log\frac{Q_{t}}{Q_{0}}dQ_{0}$ is the
entropy with respect to the stationary distribution. Thus, the dependence
of the mixing time on the starting distribution goes down from $\log(1/d(Q_{0},Q))$
to $\log\log(1/d(Q_{0},Q)))$. Moreover, just as in the case of the
Cheeger constant, for the ball walk, the Markov chain parameter is
determined by the log-Sobelov constant $\rho_{p}$ for sampling from
the density $p$ (See Theorem \ref{thm:speedy}). It is thus natural
to ask for the best possible bound on $\rho_{p}$ or $\kappa_{p}$.
Unlike the Cheeger constant, which is conjectured to be at least a
constant for isotropic logconcave densities, it is known that $\rho_{p}$
cannot be bounded from below by a universal constant, in particular
for distributions that are not ``$\psi_{2}$'' (distributions with
sub-Gaussian tail).

Kannan, Lovász and Montenegro \cite{KannanLM06} gave the following
bound on $\kappa_{p}$. Our main result (Theorem \ref{thm:logsob})
is an improvement of this bound to the best possible when the distribution
is isotropic. 
\begin{thm}[\cite{KannanLM06}]
For a logconcave density $K\subset\R^{n}$ with support of diameter
$D$, $we$ have $\kappa_{p}\gtrsim\frac{1}{D}$ and $\rho_{p}\gtrsim\frac{1}{D^{2}}$.
\end{thm}

From the above bound, it follows that the ball walk mixes in $O\left(n^{2}D^{2}\right)$
\emph{proper} steps of step size $\delta=\Theta(\frac{1}{\sqrt{n}})$.
A proper step is one where the current point changes. For an isotropic
logconcave density $\delta=\Theta(\frac{1}{\sqrt{n}})$ is small enough
so that the number of wasted steps is of the same order as the number
of proper steps in expectation. Moreover, by restricting to a ball
of radius $D=O(\sqrt{n})$, the resulting distribution remains near-isotropic
and very close in total variation distance to the original. Together,
this considerations imply a bound of $O^{*}(n^{3})$ proper steps
from any starting point as shown in \cite{KannanLM06}. Is this bound
the best possible? From a warm start, the KLS conjecture implies a
bound of $O^{*}(n^{2})$ steps and current best bound is $O^{*}(n^{2.5})$.
Thus, the mixing of the ball walk, which was the primary motivation
for formulation of the KLS conjecture, also provides a compelling
reason to study the log-Sobolev constant. Estimating the log-Sobolev
constant was posed as an open problem by Frieze and Kannan \cite{frieze1999}
when they analyzed the log-Sobolev constant of the grid walk to sample
sufficiently smooth logconcave densities.

\subsection{Results}

We prove the following bound, conjectured in this form in \cite{Vem-notes15}. 
\begin{thm}
\label{thm:n14} For any logconcave density $p$ in $\R^{n}$, with
covariance matrix $A$,
\[
\psi_{p}=O\left(\tr\left(A^{2}\right)\right)^{1/4}.
\]
\end{thm}

For isotropic $p$, this gives a bound of $\psi_{p}\lesssim n^{\frac{1}{4}}$,
improving on the current best bound. The following corollary is immediate.
We note that it also gives an alternative proof of the central limit
theorem for logconcave distributions, via Bobkov's theorem \cite{Bobkov2007}.
\begin{cor}
For any logconcave density $p$ in $\R^{n}$, the isotropic (slicing)
constant $L_{p}$ and the thin-shell constant $\sigma_{p}$ are bounded
by $O\left(n^{1/4}\right)$.
\end{cor}

We mention an algorithmic consequence. The ball walk in a convex body
$K\subseteq\R^{n}$ starts at some point $x_{0}$ in its interior
and at each step picks a uniform random point in the ball of fixed
radius $\delta$ centered at the current point, and goes to the point
if it lies in $K$. The process converges to the uniform distribution
over $K$ in the limit. Understanding the precise rate of convergence
is a major open problem with a long line of work and directly motivated
the KLS conjecture \cite{LS90,LS92,LS93,KLS95,KLS97,LV06,LV07}. Our
improvement for the KLS constant gives the following bound on the
rate of convergence. 
\begin{cor}
The mixing time of the ball walk to sample from an isotropic logconcave
density from a warm start is $O^{*}\left(n^{2.5}\right)$.
\end{cor}

Turning to the log-Sobolev constant, our main theorem is the following.
\begin{thm}
\label{thm:logsob}For any isotropic logconcave density $p$ with
support of diameter $D$, the log-Cheeger constant satisfies $\kappa_{p}\gtrsim\frac{1}{\sqrt{D}}$
and the log-Sobolev constant satisfies $\rho_{p}\gtrsim\frac{1}{D}$. 
\end{thm}

As we show in Section \ref{sec:tight}, these bounds are the best
possible (Lemma \ref{lem:lower_bound_logsob}). The improved bound
has interesting consequences. The first is an improved concentration
of mass inequality for logconcave densities. In particular, this gives
an alternative proof of Paouris' (optimal) inequality \cite{Paouris2006}
for the large deviation case ($t\gtrsim\sqrt{n})$.
\begin{thm}
\label{thm:conc}For any $L$-Lipschitz function $g$ in $\Rn$ and
any isotropic logconcave density $p$, we have that
\[
\P_{x\sim p}\left(\left|g(x)-\bar{g}(x)\right|\geq L\cdot t\right)\leq\exp(-\frac{c\cdot t^{2}}{t+\sqrt{n}})
\]
where $\bar{g}$ is the median or mean of $g(x)$ for $x\sim p$. 
\end{thm}

For the Euclidean norm, this gives 

\[
\P_{x\sim p}\left(\norm x\geq t\cdot\sqrt{n}\right)\leq\exp(-c\cdot\min\left\{ t,t^{2}\right\} \sqrt{n}).
\]
As mentioned earlier, the previous best bound was $\exp(-c\cdot\min\left\{ t,t^{3}\right\} \sqrt{n})$
\cite{GuedonM11} for the Euclidean norm and $\exp(-\frac{t}{n^{1/4}})$
for a general Lipschitz function $g$ using our improved bound on
the Cheeger constant. The new bound can be viewed as an improvement
and generalization of both. Also this concentration result does not
need bounded support for the density $p$. 

Next we bound the small ball probability in terms of the Cheeger constant.
\begin{thm}
\label{thm:small-ball}Assume that for $k\geq1$, $\psi_{p}=O\left(\tr A^{k}\right)^{\frac{1}{2k}}$
for all logconcave distribution $p$ with covariance $A$. Then, for
any isotropic logconcave distribution $p$ and any $0\leq\varepsilon\leq c_{1}$,
we have that
\[
\P_{x\sim p}\left(\norm x_{2}^{2}\leq\varepsilon n\right)\leq\varepsilon^{c_{2}k^{-1}n^{1-1/k}}
\]
for some universal constant $c_{1}$ and $c_{2}$.
\end{thm}

\begin{rem*}
Our KLS estimate verifies the case $k=2$ and gives $\P_{x\sim p}\left(\norm x_{2}^{2}\leq\varepsilon n\right)\leq\varepsilon^{c_{2}\sqrt{n}}$,
which recovers the current best small ball estimate, also due to Paouris
\cite{paouris2012small}.
\end{rem*}
As another consequence, we circle back to the analysis of the ball
walk to resolve the open problem posed by Frieze and Kannan \cite{frieze1999}.
\begin{thm}
\label{thm:speedy}The ball walk with step size $\delta=\Theta(1/\sqrt{n})$
applied to an isotropic convex body $K$ in $\R^{n}$ with support
of diameter $D$ mixes in $O(n^{2}D\log\log D)$ proper steps from
any starting point $x$ that is $n^{-O(1)}$ far from the boundary
of $K$. 
\end{thm}

The choice of $\delta=\Theta\left(1/\sqrt{n}\right)$ is the best
possible for isotropic logconcave distributions (Lemma \ref{lem:lower_bound_ball}).
The bound on the number of steps improves on the previous best bound
of $O^{*}(n^{2}D^{2})$ proper steps for the mixing of the ball walk
from an arbitrary starting point \cite{KannanLM06} and as we show
in Section \ref{sec:tight}, $O(n^{2}D)$ is the best possible bound.
For sampling, we can restrict the density to a ball of radius $O(\sqrt{n})$
losing only a negligibly small measure, so the bound is $O(n^{2.5})$
from an arbitrary starting point, which matches the current best bound
from a warm start. The mixing time from a warm start for an isotropic
logconcave density is $O(n^{2}\psi_{p}^{2})$, or $O(n^{2})$ if the
KLS conjecture is true; but from an arbitrary start, our analysis
is essentially the best possible, independent of any further progress
on the conjecture! 

\subsection{Approach for Cheeger constant: stochastic localization}

The KLS conjecture is true for Gaussian distributions. More generally,
for any distribution whose density function is the product of the
Gaussian density for $N(0,\sigma^{2}I)$ and any logconcave function,
it is known that the expansion is $\text{\ensuremath{\Omega}(1/\ensuremath{\sigma})}$
\cite{CV2014}. This fact is used crucially in the Gaussian cooling
algorithm of \cite{CV2015} for computing the volume of a convex body
by starting with a standard Gaussian restricted to a convex body and
gradually making the variance of the Gaussian large enough that it
is effectively uniform over the convex body of interest. Our overall
strategy is similar in spirit --- we start with an arbitrary isotropic
logconcave density and gradually introduce a Gaussian term in the
density of smaller and smaller variance. The isoperimetry of the resulting
distribution after sufficient time will be very good since it has
a large Gaussian factor. And crucially, it can be related to the isoperimetry
of the initial distribution. To achieve the latter, we would like
to maintain the measure of a fixed subset close to its initial value
as the distribution changes. For this, our proof uses the localization
approach to proving high-dimensional inequalities \cite{LS93,KLS95},
and in particular, the elegant stochastic version introduced by Eldan
\cite{Eldan2013} and used in subsequent papers \cite{EldanL13,EldanL15}. 

We fix a subset $E$ of the original space with measure one half according
to the original logconcave distribution (it suffices to consider such
subsets to bound the isoperimetric constant). In standard localization,
we then repeatedly bisect space using a hyperplane that preserves
the volume fraction of $E$. The limit of this process is a partition
into 1-dimensional logconcave measures (``needles''), for which
inequalities are much easier to prove. This approach runs into major
difficulties for proving the KLS conjecture. While the original measure
might be isotropic, the 1-dimensional measures could, in principle,
have variances roughly equal to the trace of the original covariance
(i.e., long thin needles), for which only much weaker inequalities
hold. Stochastic localization can be viewed as the continuous time
version of this process, where at each step, we pick a random direction
and multiply the current density with a linear function along the
chosen direction. Over time, the density can be viewed as a spherical
Gaussian times a logconcave function, with the Gaussian gradually
reducing in variance. When the Gaussian becomes sufficiently small
in variance, then the overall distribution has good isoperimetric
coefficient, determined by the inverse of the Gaussian standard deviation
(such an inequality can be shown using standard localization, as in
\cite{CV2014}). An important property of the infinitesimal change
at each step is \emph{balance -- }the density at time $t$ is a martingale
and therefore the expected measure of any subset is the same as the
original measure. Over time, the measure of a set $E$ is a random
quantity that deviates from its original value of $\frac{1}{2}$ over
time. The main question then is: what direction to use at each step
so that (a) the measure of $E$ remains bounded and (b) the Gaussian
part of the density has small variance. We show that the simplest
choice, namely a pure random direction chosen from the uniform distribution
suffices. The analysis needs a potential function that grows slowly
but still maintains good control over the spectral norm of the current
covariance matrix. The direct choice of $\norm{A_{t}}_{\spe}$, where
$A_{t}$ is the covariance matrix of the distribution at time $t$,
is hard to control. We use $\tr(A_{t}^{2})$$.$ This gives us the
improved bound of $O(n^{1/4})$.

\subsection{Approach for log-Sobolev constant: Stieltjes potential }

We apply the localization process to bound the log-Sobolev constant.
However, unlike in the case of the Cheeger constant, we cannot simply
work with subsets of measure $1/2$ or a constant; it is crucial to
consider arbitrarily small subsets. For the Cheeger constant, the
spectral norm of the covariance the $\norm{A_{t}}_{\spe}$ is bounded
via a potential function of the form $\tr\left(A_{t}^{2}\right)$.
However, for the log-Sobolev constant, to obtain a tight result without
extraneous logarithmic factors, we study the Stieltjes-type potential
$\tr\left((uI-A_{t})^{-q}\right)$. 

To define the potential, fix integer $q\geq1$ and a positive number
$\Phi>0$. Let $u(X)$ be the real-valued function on $n\times n$
symmetric matrices defined by the solution of the following equation
\begin{equation}
\tr((uI-X)^{-q})=\Phi\text{ and }X\preceq uI\label{eq:def_u}
\end{equation}
Note that this is the same as the solution to $\sum_{i=1}^{n}\frac{1}{(u-\lambda_{i})^{q}}=\Phi$
and $\lambda_{i}\leq u$ for all $i$ where $\lambda_{i}$ are the
eigenvalues of $X$. Similar potentials have been used to analyze
empirical covariance estimation \cite{srivastava2013covariance},
to build graph sparsifiers \cite{BSS12,allen2015spectral,lee2015constructing,lee2017sdp}
and to solve bandit problems \cite{audibert2013regret}.

The proof has the following ingredients: 
\begin{enumerate}
\item We show that for time $t$ up to $O(n^{-\frac{1}{2}})$, the spectral
norm of the covariance stays bounded (by a constant, say $2$) with
large probability (Lemma \ref{lem:norm_At}). This requires the use
of the Stieltjes-type potential function. 
\item Then we consider any measurable subset $S$, with $g_{0}=p_{0}(S)$
and analyze its measure at time $t$, i.e., $g_{t}=p_{t}(S)$. In
particular we show that up to time $\left(\log g_{0}+D\right)^{-1}$,
the expectation of $g_{t}\sqrt{\log(1/g_{t})}$ remains large, i.e.,
a constant factor times its initial value (Lemma \ref{lem:g_sqrt_g}). 
\item The density at time $t$ has a Gaussian component of variance $1/t$.
For such a distribution, the log-Cheeger constant is $\Omega(\sqrt{t})$
(Theorem \ref{thm:Gaussian-iso}).
\end{enumerate}
Together these facts will imply the main theorem. 

\section{Preliminaries}

In this section, we review some basic definitions and theorems that
we use. 

\subsection{Stochastic calculus}

In this paper, we only consider stochastic processes given by stochastic
differential equations. Given real-valued stochastic processes $x_{t}$
and $y_{t}$, the quadratic variations $[x]_{t}$ and $[x,y]_{t}$
are real-valued stochastic processes defined by
\[
[x]_{t}=\lim_{|P|\rightarrow0}\sum_{n=1}^{\infty}\left(x_{\tau_{n}}-x_{\tau_{n-1}}\right)^{2}\quad\text{and}\quad[x,y]_{t}=\lim_{|P|\rightarrow0}\sum_{n=1}^{\infty}\left(x_{\tau_{n}}-x_{\tau_{n-1}}\right)\left(y_{\tau_{n}}-y_{\tau_{n-1}}\right),
\]
where $P=\{0=\tau_{0}\leq\tau_{1}\leq\tau_{2}\leq\cdots\uparrow t\}$
is a stochastic partition of the non-negative real numbers, $|P|=\max_{n}\left(\tau_{n}-\tau_{n-1}\right)$
is called the \emph{mesh} of $P$ and the limit is defined using convergence
in probability. Note that $[x]_{t}$ is non-decreasing with $t$ and
$[x,y]_{t}$ can be defined via polarization as
\[
[x,y]_{t}=\frac{1}{4}\left([x+y]_{t}-[x-y]_{t}\right).
\]
For example, if the processes $x_{t}$ and $y_{t}$ satisfy the SDEs
$dx_{t}=\mu(x_{t})dt+\sigma(x_{t})dW_{t}$ and $dy_{t}=\nu(y_{t})dt+\eta(y_{t})dW_{t}$
where $W_{t}$ is a Wiener process, we have that $[x]_{t}=\int_{0}^{t}\sigma^{2}(x_{s})ds$
and $[x,y]_{t}=\int_{0}^{t}\sigma(x_{s})\eta(y_{s})ds$ and $d[x,y]_{t}=\sigma(x_{t})\eta(y_{t})dt$;
for a vector-valued SDE $dx_{t}=\mu(x_{t})dt+\Sigma(x_{t})dW_{t}$
and $dy_{t}=\nu(y_{t})dt+M(y_{t})dW_{t}$, we have that $[x^{i},x^{j}]_{t}=\int_{0}^{t}(\Sigma(x_{s})\Sigma^{T}(x_{s}))_{ij}ds$
and $d[x^{i},y^{j}]_{t}=(\Sigma(x_{t})M^{T}(y_{t}))_{ij}dt$.
\begin{lem}[Itô's formula]
\label{lem:Ito} Let $x$ be a semimartingale and $f$ be a twice
continuously differentiable function, then
\[
df(x_{t})=\sum_{i}\frac{df(x_{t})}{dx^{i}}dx^{i}+\frac{1}{2}\sum_{i,j}\frac{d^{2}f(x_{t})}{dx^{i}dx^{j}}d[x^{i},x^{j}]_{t}.
\]
\end{lem}

The next two lemmas are well-known facts about Wiener processes; first
the reflection principle.
\begin{lem}[Reflection principle]
\label{lem:reflection}Given a Wiener process $W(t)$ and $a,t\geq0$,
then we have that
\[
\P(\sup_{0\leq s\leq t}W(s)\geq a)=2\P(W(t)\geq a).
\]
\end{lem}

Second, a decomposition lemma for continuous martingales.
\begin{thm}[Dambis, Dubins-Schwarz theorem]
\label{thm:Dubins}Every continuous local martingale $M_{t}$ is
of the form
\[
M_{t}=M_{0}+W_{[M]_{t}}\text{ for all }t\geq0
\]
where $W_{s}$ is a Wiener process. 
\end{thm}

\subsection{Logconcave functions}
\begin{lem}[Dinghas; Prékopa; Leindler]
\label{lem:marginal} The convolution of two logconcave functions
is also logconcave; in particular, any linear transformation or any
marginal of a logconcave density is logconcave.
\end{lem}

The next lemma about logconcave densities is folklore, see e.g., \cite{Lovasz2007}.
\begin{lem}[Logconcave moments]
\label{lem:lcmom} Given a logconcave density $p$ in $\R^{n}$,
and any $k\geq1$, 
\[
\E_{x\sim p}\norm x^{k}\le(2k)^{k}\left(\E_{x\sim p}\norm x^{2}\right)^{k/2}.
\]
\end{lem}

The following elementary concentration lemma is also well-known (this
version is from \cite{Lovasz2007}).
\begin{lem}[Logconcave concentration]
\label{lem:log-conc} For any isotropic logconcave density $p$ in
$\R^{n},$ and any $t>0,$
\[
\P_{x\sim p}\left(\norm x>t\sqrt{n}\right)\le e^{-t+1}.
\]
\end{lem}

To prove a lower bound on the expansion, it suffices to consider subsets
of measure $1/2.$ This follows from the concavity of the isoperimetric
profile. We quote a theorem from \cite[Theorem 1.8]{Milman2009},
which applies even more generally to Riemannian manifolds under suitable
convexity-type assumptions.
\begin{thm}
\label{thm:milman}The Cheeger constant of any logconcave density
is achieved by a subset of measure $1/2.$
\end{thm}

\section{Eldan's stochastic localization}

In this section, we consider a variant of the stochastic localization
scheme introduced in \cite{Eldan2013}. In discrete localization,
the idea would be to restrict the distribution with a random halfspace
and repeat this process. In stochastic localization, this discrete
step is replaced by infinitesimal steps, each of which is a renormalization
with a linear function in a random direction. One might view this
informally as an averaging over infinitesimal needles. The discrete
time equivalent would be $p_{t+1}(x)=p_{t}(x)(1+\sqrt{h}(x-\mu_{t})^{T}w)$
for a sufficiently small $h$ and random Gaussian vector $w$. Using
the approximation $1+y\sim e^{y-\frac{1}{2}y^{2}}$, we see that over
time this process introduces a negative quadratic factor in the exponent,
which will be the Gaussian factor. As time tends to $\infty$, the
distribution tends to a more and more concentrated Gaussian and eventually
a delta function, at which point any subset has measure either $0$
or 1. The idea of the proof is to stop at a time that is large enough
to have a strong Gaussian factor in the density, but small enough
to ensure that the measure of a set is not changed by more than a
constant. 

\subsection{The process and its basic properties}

Given a distribution with logconcave density $p(x)$, we start at
time $t=0$ with this distribution and at each time $t>0$, we apply
an infinitesimal change to the density. This is done by picking a
random direction from a standard Gaussian.

In order to construct the stochastic process, we assume that the support
of $p$ is contained in a ball of radius $R>n$. There is only exponentially
small probability outside this ball, at most $e^{-cR}$ by Lemma \ref{lem:log-conc}.
Moreover, since by Theorem \ref{thm:milman}, we only need to consider
subsets of measure $1/2,$ this truncation does not affect the KLS
constant of the distribution. 
\begin{defn}
\label{def:A}Given a logconcave distribution $p$, we define the
following stochastic differential equation:
\begin{equation}
c_{0}=0,\quad dc_{t}=dW_{t}+\mu_{t}dt,\label{eq:dBt}
\end{equation}
where the probability distribution $p_{t}$, the mean $\mu_{t}$ and
the covariance $A_{t}$ are defined by
\[
p_{t}(x)=\frac{e^{c_{t}^{T}x-\frac{t}{2}\norm x_{2}^{2}}p(x)}{\int_{\Rn}e^{c_{t}^{T}y-\frac{t}{2}\norm y_{2}^{2}}p(y)dy},\quad\mu_{t}=\E_{x\sim p_{t}}x,\quad A_{t}=\E_{x\sim p_{t}}(x-\mu_{t})(x-\mu_{t})^{T}.
\]
\end{defn}

Since $\mu_{t}$ is a bounded function that is Lipschitz with respect
to $c$ and $t$, standard theorems (e.g. \cite[Sec 5.2]{oksendal2013stochastic})
show the existence and uniqueness of the solution in time $[0,T]$
for any $T>0$. We defer all proofs for statements in this section,
considered standard in stochastic calculus, to Section \ref{sec:Localization-proofs}.
Now we proceed to analyzing the process and how its parameters evolve.
Roughly speaking, the first lemma below says that the stochastic process
is the same as continuously multiplying $p_{t}(x)$ by a random infinitesimally
small linear function. 
\begin{lem}
\label{lem:def-pt}We have that $dp_{t}(x)=(x-\mu_{t})^{T}dW_{t}p_{t}(x)$
for any $x\in\Rn$, 
\end{lem}

By considering the derivative $d\log p_{t}(x)$, we see that applying
$dp_{t}(x)$ as in the lemma above results in the distribution $p_{t}(x)$,
with the Gaussian term in the density:
\begin{align*}
d\log p_{t}(x) & =\frac{dp_{t}(x)}{p_{t}(x)}-\frac{1}{2}\frac{d[p_{t}(x)]_{t}}{p_{t}(x)^{2}}=(x-\mu_{t})^{T}dW_{t}-\frac{1}{2}(x-\mu_{t})^{T}(x-\mu_{t})dt\\
 & =x^{T}dc_{t}-\frac{1}{2}\norm x^{2}dt+g(t)
\end{align*}
where the last term is independent of $x$ and the first two terms
explain the form of $p_{t}(x)$ and the appearance of the Gaussian. 

Next we analyze the change of the covariance matrix.
\begin{lem}
\label{lem:dA}We have that $dA_{t}=\int_{\Rn}(x-\mu_{t})(x-\mu_{t})^{T}\left((x-\mu_{t})^{T}dW_{t}\right)p_{t}(x)dx-A_{t}^{2}dt.$
\end{lem}

\subsection{Bounding expansion }

Our goal is to bound the expansion by the spectral norm of the covariance
matrix at time $t.$ First, we bound the measure of a set of initial
measure $\frac{1}{2}$. 
\begin{lem}
\label{lem:volumeKLS}For any set $E\subset\Rn$ with $\int_{E}p(x)dx=\frac{1}{2}$
and $t\geq0$, we have that
\[
\P\left(\frac{1}{4}\leq\int_{E}p_{t}(x)dx\leq\frac{3}{4}\right)\geq\frac{9}{10}-\P\left(\int_{0}^{t}\norm{A_{s}}_{\spe}ds\geq\frac{1}{64}\right).
\]
\end{lem}

\begin{proof}
Let $g_{t}=\int_{E}p_{t}(x)dx$. Then, we have that
\[
dg_{t}=\left\langle \int_{E}(x-\mu_{t})p_{t}(x)dx,dW_{t}\right\rangle 
\]
where the integral might not be $0$ because it is over the subset
$E$ and not all of $\R^{n}$. Hence, we have
\begin{align*}
d[g_{t}]_{t} & =\norm{\int_{E}(x-\mu_{t})p_{t}(x)dx}_{2}^{2}dt=\max_{\norm{\zeta}_{2}\leq1}\left(\int_{E}\zeta^{T}(x-\mu_{t})p_{t}(x)dx\right)^{2}dt\\
 & \leq\left(\max_{\norm{\zeta}_{2}\leq1}\int_{\Rn}\left(\zeta^{T}(x-\mu_{t})\right)^{2}p_{t}(x)dx\right)\left(\int_{\Rn}p_{t}(x)dx\right)dt\\
 & =\max_{\norm{\zeta}_{2}\leq1}\left(\zeta^{T}A_{t}\zeta\right)dt=\norm{A_{t}}_{\spe}dt.
\end{align*}
Hence, we have that $\frac{d[g_{t}]_{t}}{dt}\leq\norm{A_{t}}_{\spe}$.
By the Dambis, Dubins-Schwarz theorem, there exists a Wiener process
$\tilde{W}_{t}$ such that $g_{t}-g_{0}$ has the same distribution
as $\tilde{W}_{[g]_{t}}$. Using $g_{0}=\frac{1}{2}$, we have that
\begin{align*}
\P(\frac{1}{4}\leq g_{t}\leq\frac{3}{4}) & =\P(\frac{-1}{4}\leq\tilde{W}_{[g]_{t}}\leq\frac{1}{4})\geq1-\P(\max_{0\leq s\leq\frac{1}{64}}\left|\tilde{W}_{s}\right|>\frac{1}{4})-\P([g]_{t}>\frac{1}{64})\\
 & \overset{\cirt 1}{\geq}1-4\P(\tilde{W}_{\frac{1}{64}}>\frac{1}{4})-\P([g]_{t}>\frac{1}{64})\\
 & \overset{\cirt 2}{\geq}\frac{9}{10}-\P([g]_{t}>\frac{1}{64})
\end{align*}
where we used reflection principle for 1-dimensional Brownian motion
in $\cirt 1$ and the concentration of normal distribution in $\cirt 2,$
namely $\P_{x\sim N(0,1)}$$\left(x>2\right)\le0.0228$. 
\end{proof}

At time $t$, the distribution is $t$-strongly log-concave and it
is known that it has expansion $\Omega(\sqrt{t})$. The following
isoperimetric inequality was proved in \cite{CV2014} and was also
used in \cite{Eldan2013}. In Theorem \ref{thm:Gaussian-iso-2}, we
prove a strengthen version of the following theorem, which is needed
for bounding the log-Sobolev constant.
\begin{thm}
\label{thm:Gaussian-iso}Let $h(x)=f(x)e^{-\frac{t}{2}\norm x_{2}^{2}}/\int f(y)e^{-\frac{t}{2}\norm y_{2}^{2}}dy$
where $f:\R^{n}\rightarrow\R_{+}$ is an integrable logconcave function.
Then $h$ is logconcave and for any measurable subset $S$ of $\R^{n}$,
\[
\int_{\partial S}h(x)dx=\Omega\left(\sqrt{t}\right)\min\left\{ \int_{S}h(x)dx,\int_{\R^{n}\setminus S}h(x)dx\right\} .
\]
In other words, the expansion of h is $\Omega\left(\sqrt{t}\right)$.
\end{thm}

We can now prove a bound on the expansion. 
\begin{lem}
\label{lem:boundAgivesKLS}Given a logconcave distribution $p$. Let
$A_{t}$ be defined by Definition \ref{def:A} using initial distribution
$p$. Suppose that there is $T>0$ such that
\[
\P\left(\int_{0}^{T}\norm{A_{s}}_{\op}ds\leq\frac{1}{64}\right)\geq\frac{3}{4}
\]
Then, we have that $\psi_{p}=O\left(T^{-1/2}\right).$
\end{lem}

\begin{proof}
By Milman's theorem \cite{Milman2009}, it suffices to consider subsets
of measure $\frac{1}{2}.$ Consider any measurable subset $E$ of
$\R^{n}$ of initial measure $\frac{1}{2}$. By Lemma \ref{lem:def-pt},
$p_{t}$ is a martingale and therefore
\[
\int_{\partial E}p(x)dx=\int_{\partial E}p_{0}(x)dx=\E\left(\int_{\partial E}p_{t}(x)dx\right).
\]
Next, by the definition of $p_{T}$ (\ref{eq:dBt}), we have that
$p_{T}(x)\propto e^{c_{T}^{T}x-\frac{T}{2}\norm x^{2}}p(x)$ and Theorem
\ref{thm:Gaussian-iso} shows that the expansion of $E$ is $\Omega\left(\sqrt{T}\right)$.
Hence, we have
\begin{align*}
\int_{\partial E}p(x)dx & =\E\int_{\partial E}p_{T}(x)dx\gtrsim\sqrt{T}\cdot\E\left(\min\left(\int_{E}p_{T}(x)dx,\int_{\bar{E}}p_{T}(x)dx\right)\right)\\
 & \gtrsim\sqrt{T}\cdot\P\left(\frac{1}{4}\leq\int_{E}p_{T}(x)dx\leq\frac{3}{4}\right)\overset{\lref{volumeKLS}}{\gtrsim}\sqrt{T}\cdot\left(\frac{9}{10}-\P(\int_{0}^{t}\norm{A_{s}}_{\op}ds\geq\frac{1}{64})\right)=\Omega(\sqrt{T})
\end{align*}
where we used the assumption at the end. Using Theorem \ref{thm:milman},
this shows that $\psi_{p}=O\left(T^{-1/2}\right).$
\end{proof}

\section{Controlling $A_{t}$ via the potential $\protect\tr(A_{t}^{2})$
\label{sec:A2}}

\subsection{Third moment bounds}

Here are two key lemmas about the third-order tensor of a log-concave
distribution. 
\begin{lem}
\label{lem:tensorestimate}Given a logconcave distribution $p$ with
mean $\mu$ and covariance $A$. For any positive semi-definite matrix
$C$, we have that
\[
\norm{\E_{x\sim p}(x-\mu)(x-\mu)^{T}C(x-\mu)}_{2}=O\left(\norm A_{\op}^{1/2}\tr\left(A^{1/2}CA^{1/2}\right)\right).
\]
\end{lem}

\begin{proof}
We first prove the case $C=vv^{T}$. Taking $y=A^{-1/2}(x-\mu)$ and
$w=A^{1/2}v$. Then, $y$ follows an isotropic logconcave distribution
$\tilde{p}$ and hence
\begin{align*}
\norm{\E_{x\sim p}(x-\mu)(x-\mu)^{T}C(x-\mu)}_{2} & =\norm{\E_{y\sim\tilde{p}}A^{1/2}y\left(y^{T}w\right)^{2}}_{2}=\max_{\norm{\zeta}_{2}\leq1}\E_{y\sim\tilde{p}}(A^{1/2}y)^{T}\zeta\left(y^{T}w\right)^{2}\\
 & \leq\max_{\norm{\zeta}_{2}\leq1}\sqrt{\E_{y\sim\tilde{p}}\left((A^{1/2}y)^{T}\zeta\right)^{2}}\sqrt{\E_{y\sim\tilde{p}}\left(y^{T}w\right)^{4}}\lesssim\norm A_{\op}^{1/2}\norm w_{2}^{2}
\end{align*}
where we used the fact that for a fixed $w$, $y^{T}w$ has a one-dimensional
logconcave distribution (Lemma \ref{lem:marginal}) and hence Lemma
\ref{lem:lcmom} shows that
\[
\E_{y\sim\tilde{p}}\left(y^{T}w\right)^{4}\lesssim\left(\E_{y\sim\tilde{p}}\left(y^{T}w\right)^{2}\right)^{2}\lesssim\norm w_{2}^{4}.
\]

For a general PSD matrix $C$, we write $C=\sum\lambda_{i}v_{i}v_{i}^{T}$
where $\lambda_{i}\ge0$, $v_{i}$ are eigenvalues and eigenvectors
of $C$. Hence, we have that
\begin{align*}
\norm{\E_{x\sim p}(x-\mu)(x-\mu)^{T}C(x-\mu)}_{2} & \leq\sum_{i}\lambda_{i}\norm{\E_{x\sim p}(x-\mu)(x-\mu)^{T}v_{i}v_{i}^{T}(x-\mu)}_{2}\lesssim\sum_{i}\lambda_{i}\norm A_{\op}^{1/2}\norm{A^{1/2}v_{i}}^{2}\\
 & =\norm A_{\op}^{1/2}\sum_{i}\tr\left(A^{1/2}\lambda_{i}v_{i}v_{i}^{T}A^{1/2}\right)=\norm A_{\op}^{1/2}\tr\left(A^{1/2}CA^{1/2}\right).
\end{align*}
\end{proof}
\begin{lem}
\label{lem:tensor_norm}Given a logconcave distribution $p$ with
mean $\mu$ and covariance $A$. We have
\[
\E_{x,y\sim p}\left|\left\langle x-\mu,y-\mu\right\rangle \right|^{3}=O\left(\tr\left(A^{2}\right)^{3/2}\right).
\]
\end{lem}

\begin{proof}
Without loss of generality, we assume $\mu=0$. For a fixed $x$ and
random $y$, $\langle x,y\rangle$ follows a one-dimensional logconcave
distribution (Lemma \ref{lem:marginal}) and hence Lemma \ref{lem:lcmom}
shows that
\[
\E_{y\sim p}\left|\langle x,y\rangle\right|^{3}\lesssim\left(\E_{y\sim p}\langle x,y\rangle^{2}\right)^{3/2}=\left(x^{T}Ax\right)^{3/2}.
\]
Next, we note that $A^{1/2}x$ follows a logconcave distribution (Lemma
\ref{lem:marginal}) and hence Lemma \ref{lem:lcmom} shows that
\[
\E_{x,y\sim p}\left|\langle x,y\rangle\right|^{3}\lesssim\E_{x\sim p}\norm{A^{1/2}x}^{3}\lesssim\left(\E_{x\sim p}\norm{A^{1/2}x}^{2}\right)^{3/2}=\tr\left(A^{2}\right)^{3/2}.
\]
\end{proof}

\subsection{Analysis of $A_{t}$}

Using Itô's formula and Lemma \ref{lem:dA}, we compute the derivatives
of $\tr A_{t}^{2}$. 
\begin{lem}
\label{lem:trace_itos}Let $A_{t}$ be defined by Definition \ref{def:A}.
We have that
\begin{align*}
d\tr A_{t}^{2}= & 2\E_{x\sim p_{t}}(x-\mu_{t})^{T}A_{t}(x-\mu_{t})(x-\mu_{t})^{T}dW_{t}-2\tr(A_{t}^{3})dt+\E_{x,y\sim p_{t}}\left((x-\mu_{t})^{T}(y-\mu_{t})\right)^{3}dt.
\end{align*}
\end{lem}

\begin{lem}
\label{lem:trace}Given a logconcave distribution $p$ with covariance
matrix $A$. Let $A_{t}$ defined by Definition \ref{def:A} using
initial distribution $p$. There is a universal constant $c_{1}$
such that
\[
\P(\max_{t\in[0,T]}\tr(A_{t}^{2})\geq8\tr(A_{0}^{2}))\leq0.01\quad\text{with}\quad T=\frac{c_{1}}{\sqrt{\tr(A_{0}^{2})}}.
\]
\end{lem}

\begin{proof}
Let $\Phi_{t}=\tr A_{t}^{2}$. By Lemma \ref{lem:trace_itos}, we
have that
\begin{align}
d\Phi_{t}= & -2\tr(A_{t}^{3})dt+\E_{x,y\sim p_{t}}\left((x-\mu_{t})^{T}(y-\mu_{t})\right)^{3}dt+2\E_{x\sim p_{t}}(x-\mu_{t})^{T}A_{t}(x-\mu_{t})(x-\mu_{t})^{T}dW_{t}\nonumber \\
\defeq & \delta_{t}dt+v_{t}^{T}dW_{t}.\label{eq:dPhi}
\end{align}

For the drift term $\delta_{t}dt$, Lemma \ref{lem:tensor_norm} shows
that
\begin{equation}
\delta_{t}\le\E_{x,y\sim p_{t}}\left((x-\mu_{t})^{T}(y-\mu_{t})\right)^{3}=O\left(\tr\left(A_{t}^{2}\right)^{3/2}\right)\leq C'\Phi_{t}^{3/2}\label{eq:alpha_bound}
\end{equation}
for some universal constant $C'$. Note that we dropped the term $-2\tr(A_{t}^{3})$
since $A_{t}$ is positive semidefinite and therefore the term is
negative.

For the martingale term $v_{t}^{T}dW_{t}$, we note that
\[
\norm{v_{t}}_{2}=\norm{\E_{x\sim p_{t}}(x-\mu_{t})^{T}A_{t}(x-\mu_{t})(x-\mu_{t})}_{2}\overset{\lref{tensorestimate}}{\leq}\norm{A_{t}}_{\op}^{1/2}\tr A_{t}^{2}\leq\Phi_{t}^{5/4}.
\]
So the drift term grows roughly as $\Phi^{3/2}t$ while the stochastic
term grows as $\Phi_{t}^{5/4}\sqrt{t}.$ Thus, both bounds (on the
drift term and the stochastic term) suggest that for $t$ up to $O\left(\frac{1}{\sqrt{\Phi_{0}}}\right)$,
the potential $\Phi_{t}$ remains $O(\Phi_{0})$. We now formalize
this, by decoupling the two terms. 

Let $f(a)=-\frac{1}{\sqrt{a+\Phi_{0}}}.$ By (\ref{eq:dPhi}) and
Itô's formula, we have that
\begin{align}
df(\Phi_{t}) & =f'(\Phi_{t})d\Phi_{t}+\frac{1}{2}f''(\Phi_{t})d[\Phi]_{t}=\left(\frac{1}{2}\frac{\delta_{t}}{(\Phi_{t}+\Phi_{0})^{3/2}}-\frac{3}{8}\frac{\norm{v_{t}}_{2}^{2}}{(\Phi_{t}+\Phi_{0})^{5/2}}\right)dt+\frac{1}{2}\frac{v_{t}^{T}dW_{t}}{(\Phi_{t}+\Phi_{0})^{3/2}}\nonumber \\
 & \leq C'dt+dY_{t}\label{eq:dpsi4}
\end{align}
where $dY_{t}=\frac{1}{2}\frac{v_{t}^{T}dW_{t}}{(\Phi_{t}+\Phi_{0})^{3/2}}$,
$Y_{t}=0$ and $C'$ is the universal constant in (\ref{eq:alpha_bound}).

Note that
\[
\frac{d[Y]_{t}}{dt}=\frac{1}{4}\frac{\norm{v_{t}}_{2}^{2}}{(\Phi_{t}+\Phi_{0})^{3}}=O(1)\frac{\Phi^{5/2}}{(\Phi_{t}+\Phi_{0})^{3}}\leq\frac{C}{\sqrt{\Phi_{0}}}.
\]
By Theorem \ref{thm:Dubins}, there exists a Wiener process $\tilde{W}_{t}$
such that $Y_{t}$ has the same distribution as $\tilde{W}_{[Y]_{t}}$.
Using the reflection principle for 1-dimensional Brownian motion,
we have that
\[
\P(\max_{t\in[0,T]}Y_{t}\geq\gamma)\leq\P(\max_{t\in[0,\frac{C}{\sqrt{\Phi_{0}}}T]}\tilde{W}_{t}\geq\gamma)=2\P(\tilde{W}_{\frac{C}{\sqrt{\Phi_{0}}}T}\geq\gamma)\leq2\exp(-\frac{\gamma^{2}\sqrt{\Phi_{0}}}{2CT}).
\]
Since $f(\Phi_{0})=-\frac{1}{\sqrt{2\Phi_{0}}}$, (\ref{eq:dpsi4})
shows that
\[
\P(\max_{t\in[0,T]}f(\Phi_{t})\geq-\frac{1}{\sqrt{2\Phi_{0}}}+C'T+\gamma)\leq2\exp(-\frac{\gamma^{2}\sqrt{\Phi_{0}}}{2CT}).
\]

Putting $T=\frac{1}{256(C'+C)\sqrt{\Phi_{0}}}$ and $\gamma=\frac{1}{4\sqrt{\Phi_{0}}}$,
we have that
\[
\P(\max_{t\in[0,T]}f(\Phi_{t})\geq-\frac{1}{3\sqrt{\Phi_{0}}})\leq2\exp(-8)).
\]
Note that $f(\Phi_{t})\geq-\frac{1}{3\sqrt{\Phi_{0}}}$ if and only
if $\Phi_{t}\geq8\Phi_{0}$. Hence, $\P(\max_{t\in[0,T]}\Phi_{t}\geq8\Phi_{0})\leq0.01.$
\end{proof}

\subsection{Proof of Theorem \ref{thm:n14}}
\begin{proof}[Proof of Theorem \ref{thm:n14}.]
By Lemma \ref{lem:trace}, we have that
\[
\P(\max_{s\in[0,t]}\tr A_{s}^{2}\leq8\tr A_{0}^{2})\geq0.99\quad\text{with}\quad t=\frac{c_{1}}{\sqrt{\tr A_{0}^{2}}}.
\]
Since $\tr A_{t}^{2}\leq8\tr A_{0}^{2}$ implies that $\norm{A_{t}}_{\op}\le\sqrt{8\tr A_{0}^{2}}$,
we have that $\P(\int_{0}^{T}\norm{A_{s}}_{\op}ds\leq\frac{1}{64})\geq0.99$
where $T=\min\left\{ \frac{1}{64\sqrt{8}},c_{1}\right\} \frac{1}{\sqrt{\tr A_{0}^{2}}}$.
Now the theorem follows from Lemma \ref{lem:boundAgivesKLS}.
\end{proof}

\subsection{Localization proofs\label{sec:Localization-proofs}}

We begin with the proof of the infinitesimal change in the density.
\begin{proof}[Proof of Lemma \ref{lem:def-pt}.]
 Let $q_{t}(x)=e^{c_{t}^{T}x-\frac{t}{2}\norm x^{2}}p(x)$. By Itô's
formula, applied to $f(a,t)\defeq e^{a-\frac{t}{2}\norm x^{2}}p(x)$
with $a=c_{t}^{T}x$, we have that
\begin{align*}
dq_{t}(x) & =\frac{df(a,t)}{da}dc_{t}^{T}x+\frac{df(a,t)}{dt}dt+\frac{1}{2}\frac{d^{2}f(a,t)}{da^{2}}d[c_{t}^{T}x]_{t}+\frac{1}{2}\frac{d^{2}f(a,t)}{dt^{2}}d[t]_{t}+\frac{1}{2}\cdot2\cdot\frac{d^{2}f(a,t)}{dadt}d[c_{t}^{T}x,t]_{t}\\
 & =\left(dc_{t}^{T}x-\frac{1}{2}\norm x_{2}^{2}dt+\frac{1}{2}d[c_{t}^{T}x]_{t}\right)q_{t}(x).
\end{align*}
By the definition of $c_{t}$, we have $dc_{t}^{T}x=\left\langle dW_{t}+\mu_{t}dt,x\right\rangle $.
The quadratic variation of $c_{t}^{T}x$ is $d[c_{t}^{T}x]_{t}=\left\langle x,x\right\rangle dt.$
The other two quadratic variation terms are zero. Therefore, this
gives 
\begin{align}
dq_{t}(x) & =\left\langle dW_{t}+\mu_{t}dt,x\right\rangle q_{t}(x).\label{eq:dq}
\end{align}

Let $V_{t}=\int_{\Rn}q_{t}(y)dy$. Then, we have
\[
dV_{t}=\int_{\Rn}dq_{t}(y)dy=\int_{\Rn}\left\langle dW_{t}+\mu_{t}dt,y\right\rangle q_{t}(y)dy=V_{t}\left\langle dW_{t}+\mu_{t}dt,\mu_{t}\right\rangle .
\]
By Itô's formula, we have that
\begin{align}
dV_{t}^{-1} & =-\frac{1}{V_{t}^{2}}dV_{t}+\frac{1}{V_{t}^{3}}d[V]_{t}=-V_{t}^{-1}\left\langle dW_{t}+\mu_{t}dt,\mu_{t}\right\rangle +V_{t}^{-1}\left\langle \mu_{t},\mu_{t}\right\rangle dt=-V_{t}^{-1}\left\langle dW_{t},\mu_{t}\right\rangle .\label{eq:dVinv}
\end{align}

Combining (\ref{eq:dq}) and (\ref{eq:dVinv}), we have that
\[
dp_{t}(x)=d(V_{t}^{-1}q_{t}(x))=q_{t}(x)dV_{t}^{-1}+V_{t}^{-1}dq_{t}(x)+d[V_{t}^{-1},q_{t}(x)]_{t}=p_{t}(x)\left\langle dW_{t},x-\mu_{t}\right\rangle .
\]
\end{proof}
The next proof is for the change in the covariance matrix.
\begin{proof}[Proof of Lemma \ref{lem:dA}.]
 Recall that
\[
A_{t}=\int_{\Rn}(x-\mu_{t})(x-\mu_{t})^{T}p_{t}(x)dx.
\]
Viewing $A_{t}=f(\mu_{t},p_{t})$, i.e., as a function of the variables
$\mu_{t}$ and $p_{t}$, we apply Itô's formula. In the derivation
below, we use $[\mu_{t},\mu_{t}^{T}]_{t}$ to denote the matrix whose
$i,j$ coordinate is $[\mu_{t,i},\mu_{t,j}]_{t}$. Similarly, $[\mu_{t},p_{t}(x)]_{t}$
is a column vector and $[\mu_{t}^{T},p_{t}(x)]_{t}$ is a row vector.
\begin{align*}
dA_{t}= & \int_{\Rn}(x-\mu_{t})(x-\mu_{t})^{T}dp_{t}(x)dx-\int_{\Rn}d\mu_{t}(x-\mu_{t})^{T}p_{t}(x)dx-\int_{\Rn}(x-\mu_{t})(d\mu_{t})^{T}p_{t}(x)dx\\
 & -\frac{1}{2}\cdot2\int_{\Rn}(x-\mu_{t})d[\mu_{t}^{T},p_{t}(x)]_{t}dx-\frac{1}{2}\cdot2\int_{\Rn}d[\mu_{t},p_{t}(x)]_{t}(x-\mu_{t})^{T}dx+\frac{1}{2}\cdot2d[\mu_{t},\mu_{t}^{T}]_{t}\int_{\Rn}p_{t}(x)dx.
\end{align*}
where the factor 2 comes from the Hessians of $x^{2}$ and $xy$.
Now the second term vanishes because
\[
\int_{\Rn}d\mu_{t}(x-\mu_{t})^{T}p_{t}(x)dx=d\mu_{t}(\int_{\Rn}(x-\mu_{t})p_{t}(x)dx)^{T}=0.
\]
Similarly, the third term also vanishes:$\int_{\Rn}(x-\mu_{t})(d\mu_{t})^{T}p_{t}(x)dx=0.$

To compute the last 3 terms, we note that
\begin{align*}
d\mu_{t} & =d\int_{\Rn}xp_{t}(x)dx=\int_{\Rn}x(x-\mu_{t})^{T}dW_{t}p_{t}(x)dx\\
 & =\int_{\Rn}(x-\mu_{t})(x-\mu_{t})^{T}dW_{t}p_{t}(x)dx+\int_{\Rn}\mu_{t}(x-\mu_{t})^{T}dW_{t}p_{t}(x)dx=A_{t}dW_{t}.
\end{align*}
Therefore, we have for the last term
\[
\left(d[\mu_{t},\mu_{t}^{T}]_{t}\right)_{ij}=\sum_{\ell}\left(A_{t}\right)_{i\ell}\left(A_{t}\right)_{j\ell}dt=(A_{t}A_{t}^{T})_{ij}dt=(A_{t}^{2})_{ij}dt
\]
which we can simply write as $d[\mu_{t},\mu_{t}^{T}]_{t}=A_{t}^{2}dt$.
Similarly, we have $d[\mu_{t},p_{t}(x)]_{t}=p_{t}(x)A_{t}(x-\mu_{t})dt.$
This gives the fourth term
\[
\int_{\Rn}(x-\mu_{t})d[\mu_{t}^{T},p_{t}(x)]_{t}dx=\int_{\Rn}(x-\mu_{t})(x-\mu_{t})^{T}A_{t}p_{t}(x)dtdx=A_{t}^{2}dt.
\]
Similarly, we have the fifth term $\int_{\Rn}d[\mu_{t},p_{t}(x)]_{t}(x-\mu_{t})^{T}dx=A_{t}^{2}dt$.
Combining all the terms, we have that
\begin{align*}
dA_{t}= & \int_{\Rn}(x-\mu_{t})(x-\mu_{t})^{T}dp_{t}(x)dx-A_{t}^{2}dt.\qedhere
\end{align*}
\end{proof}
Next is the proof of stochastic derivative of the potential $\Phi_{t}=\tr(A_{t}^{2})$.
\begin{proof}[Proof of Lemma \ref{lem:trace_itos}]
Let $\Phi(X)=\tr(X^{2})$. Then the first and second-order directional
derivatives of $\Phi$ at $X$ is given by $\left.\frac{\partial\Phi}{\partial X}\right|_{H}=2\tr(XH)$
and $\left.\frac{\partial^{2}\Phi}{\partial X\partial X}\right|_{H_{1},H_{2}}=2\tr(H_{1}H_{2})$.
Using these and Itô's formula, we have that 
\[
d\tr(A_{t}^{2})=2\tr(A_{t}dA_{t})+\sum_{ij}d[A_{ij},A_{ji}]_{t}
\]
 where $A_{ij}$ is the real-valued stochastic process defined by
the $(i,j)^{th}$ entry of $A_{t}$. Using Lemma \ref{lem:dA} and
Lemma \ref{lem:def-pt}, we have that

\begin{align}
dA_{t} & =\sum_{z}\E_{x\sim p_{t}}(x-\mu_{t})(x-\mu_{t})^{T}(x-\mu_{t})_{z}dW_{t,z}-A_{t}^{2}dt\label{eq:dAz2_gen}
\end{align}
where $W_{t,z}$ is the $z^{th}$ coordinate of $W_{t}$. Therefore,
\begin{align}
d[A_{ij},A_{ji}]_{t} & =\sum_{z}\left(\E_{x\sim p_{t}}(x-\mu_{t})_{i}(x-\mu_{t})_{j}(x-\mu_{t})^{T}e_{z}\right)\left(\E_{x\sim p_{t}}(x-\mu_{t})_{j}(x-\mu_{t})_{i}(x-\mu_{t})^{T}e_{z}\right)dt\nonumber \\
 & =\E_{x,y\sim p_{t}}(x-\mu_{t})_{i}(x-\mu_{t})_{j}(y-\mu_{t})_{j}(y-\mu_{t})_{i}(x-\mu_{t})^{T}(y-\mu_{t})dt.\label{eq:dAv2_gen}
\end{align}
Using the formula for $dA_{t}$ (\ref{eq:dAz2_gen}) and $d[A_{ij},A_{ji}]_{t}$
(\ref{eq:dAv2_gen}), we have that
\begin{align*}
d\tr(A_{t}^{2})= & 2\E_{x\sim p_{t}}(x-\mu_{t})^{T}A_{t}(x-\mu_{t})(x-\mu_{t})^{T}dW_{t}-2\tr(A_{t}^{3})dt\\
 & +\sum_{ij}\E_{x,y\sim p_{t}}(x-\mu_{t})_{i}(x-\mu_{t})_{j}(y-\mu_{t})_{j}(y-\mu_{t})_{i}(x-\mu_{t})^{T}(y-\mu_{t})dt\\
= & 2\E_{x\sim p_{t}}(x-\mu_{t})^{T}A_{t}(x-\mu_{t})(x-\mu_{t})^{T}dW_{t}-2\tr(A_{t}^{3})dt+\E_{x,y\sim p_{t}}((x-\mu_{t})^{T}(y-\mu_{t}))^{3}dt.\qedhere
\end{align*}
\end{proof}

\section{Controlling $A_{t}$ via a Stieltjes potential}

The goal of this section is to prove the following theorem, which
implies Theorem \ref{thm:logsob}.
\begin{thm}
\label{lem:expand_p}Given an isotropic logconcave distribution $p$
with support of diameter $D$. Then, for any measurable subset $S$,
\[
p(\partial S)\geq\Omega\left(\frac{\log\frac{1}{p(S)}}{D}+\sqrt{\frac{\log\frac{1}{p(S)}}{D}}\right)p(S).
\]
\end{thm}

\subsection{Bounding the spectral norm of the covariance matrix}

\label{subsec:Bounding-the-spectral}The main lemma of this section
is the following. 
\begin{lem}
\label{lem:norm_At}Assume that for $k\geq1$, $\psi_{p}\lesssim\left(\tr A^{k}\right)^{1/2k}$
for all logconcave distribution $p$ with covariance $A$. There is
some universal constant $c\geq0$ such that for any $0\leq T\leq\frac{1}{c\cdot kn^{1/k}}$,
we have that
\[
\P(\max_{t\in[0,T]}\norm{A_{t}}_{\spe}\geq2)\leq2\exp(-\frac{1}{cT}).
\]
\end{lem}

\begin{rem*}
Theorem \ref{thm:n14} proved that the assumption holds with $k=2$.
\end{rem*}
We defer the proof of the following lemma to the end of this section. 
\begin{lem}
\label{lem:stoc_du}Let $u_{t}=u(A_{t})$ (see Definition \ref{eq:def_u}).
We have
\[
du(A_{t})=\alpha_{t}^{T}dW_{t}+\beta_{t}dt
\]
where
\begin{align*}
\alpha_{t}= & \frac{1}{\kappa_{t}}\E_{x\sim\tilde{p}_{t}}x^{T}(uI-A_{t})^{-(q+1)}x\cdot x,\\
\frac{\beta_{t}}{q+1}\leq & \frac{1}{2\kappa_{t}}\E_{x,y\sim\tilde{p}_{t}}x^{T}(uI-A_{t})^{-1}y\cdot x^{T}(uI-A_{t})^{-(q+1)}y\cdot x^{T}y\\
 & -\frac{1}{\kappa_{t}^{2}}\E_{x,y\sim\tilde{p}_{t}}x^{T}(uI-A_{t})^{-(q+1)}x\cdot y^{T}(uI-A_{t})^{-(q+2)}y\cdot x^{T}y\\
 & +\frac{1}{2\kappa_{t}^{3}}\tr((uI-A_{t})^{-(q+2)})\cdot\E_{x,y\sim\tilde{p}_{t}}x^{T}(uI-A_{t})^{-(q+1)}x\cdot y^{T}(uI-A_{t})^{-(q+1)}y\cdot x^{T}y,\\
\kappa_{t}= & \tr((uI-A_{t})^{-(q+1)})
\end{align*}
and $\tilde{p}_{t}$ be the translation of $p_{t}$ defined by $\tilde{p}_{t}(x)=p_{t}(x+\mu_{t})$.
\end{lem}

To estimate $\alpha_{t}$, we can use Lemma \ref{lem:tensorestimate}.
To estimate $\beta_{t}$, we prove the following bound that crucially
uses the KLS bound for non-isotropic logconcave distribution (the
assumption in Lemma \ref{lem:norm_At}).
\begin{lem}
\label{lem:tensor_bound}Under the assumption as Lemma \ref{lem:norm_At}.
Given a logconcave distribution $p$ with mean $\mu$ and covariance
$A$. For any $B^{(1)},B^{(2)},B^{(3)}\succeq0$, 
\begin{align*}
 & \left|\E_{x,y\sim p}(x-\mu)^{T}B^{(1)}(y-\mu)\cdot(x-\mu)^{T}B^{(2)}(y-\mu)\cdot(x-\mu)^{T}B^{(3)}(y-\mu)\right|\\
\leq & O(1)\cdot\tr(A^{\frac{1}{2}}B^{(1)}A^{\frac{1}{2}})\cdot\left(\tr(A^{\frac{1}{2}}B^{(2)}A^{\frac{1}{2}})^{k}\right)^{1/k}\cdot\norm{A^{\frac{1}{2}}B^{(3)}A^{\frac{1}{2}}}_{\spe}.
\end{align*}
\end{lem}

\begin{proof}
Without loss of generality, we can assume $p$ is isotropic. Furthermore,
we can assume $B^{(1)}$ is diagonal. Let $\Delta^{(k)}=\E_{x\sim p}xx^{T}\cdot x^{T}e_{k}$.
Then, we have that 
\begin{align}
\E_{x,y\sim p}x^{T}B^{(1)}y\cdot x^{T}B^{(2)}y\cdot x^{T}B^{(3)}y & =\sum_{k}B_{kk}^{(1)}\tr(B^{(2)}\Delta^{(k)}B^{(3)}\Delta^{(k)})\nonumber \\
 & \leq\norm{B^{(3)}}_{\spe}\sum_{k}B_{kk}^{(1)}\tr(\Delta^{(k)}B^{(2)}\Delta^{(k)}).\label{eq:tensor_bound}
\end{align}
Note that
\begin{align*}
\tr(\Delta^{(k)}B^{(2)}\Delta^{(k)}) & =\E_{x\sim p}x^{T}B^{(2)}\Delta^{(k)}x\cdot x_{k}\\
 & \leq\sqrt{\E x_{k}^{2}}\sqrt{\Var_{x\sim p}x^{T}B^{(2)}\Delta^{(k)}x}\\
 & =\sqrt{\Var_{y\sim\tilde{p}}y^{T}\left(B^{(2)}\right)^{\frac{1}{2}}\Delta^{(k)}\left(B^{(2)}\right)^{-\frac{1}{2}}y}
\end{align*}
where $\tilde{p}$ is the distribution given by $\left(B^{(2)}\right)^{\frac{1}{2}}x$
where $x\sim p$. Note that $\tilde{p}$ is a logconcave distribution
with mean $0$ and covariance $B^{(2)}$. Theorem \ref{thm:Poincare}
together with the assumption that $\psi_{p}\lesssim\left(\tr A^{k}\right)^{-1/2k}$,
we have
\[
\Var_{y\sim\tilde{p}}y^{T}\left(B^{(2)}\right)^{\frac{1}{2}}\Delta^{(k)}\left(B^{(2)}\right)^{-\frac{1}{2}}y\apprle\left(\tr(B^{(2)})^{k}\right)^{1/k}\cdot\E_{y\sim\tilde{p}}\norm{\left(B^{(2)}\right)^{\frac{1}{2}}\Delta^{(k)}\left(B^{(2)}\right)^{-\frac{1}{2}}y}^{2}
\]
Hence, we have that
\begin{align*}
\tr(\Delta^{(k)}B^{(2)}\Delta^{(k)}) & \leq\left(\tr(B^{(2)})^{k}\right)^{1/2k}\sqrt{\E_{y\sim\tilde{p}}\norm{\left(B^{(2)}\right)^{\frac{1}{2}}\Delta^{(k)}\left(B^{(2)}\right)^{-\frac{1}{2}}y}^{2}}\\
 & =\left(\tr(B^{(2)})^{k}\right)^{1/2k}\sqrt{\E_{y\sim\tilde{p}}\tr\left(B^{(2)}\right)^{-\frac{1}{2}}\Delta^{(k)}B^{(2)}\Delta^{(k)}\left(B^{(2)}\right)^{-\frac{1}{2}}yy^{T}}\\
 & =\left(\tr(B^{(2)})^{k}\right)^{1/2k}\sqrt{\tr\Delta^{(k)}B^{(2)}\Delta^{(k)}}.
\end{align*}
Hence, we have that
\[
\tr(\Delta^{(k)}B^{(2)}\Delta^{(k)})\leq\left(\tr(B^{(2)})^{k}\right)^{1/k}.
\]
Putting it into (\ref{eq:tensor_bound}) gives that
\[
\left|\E_{x,y\sim p}x^{T}B^{(1)}y\cdot x^{T}B^{(2)}y\cdot x^{T}B^{(3)}y\right|\apprle\tr B^{(1)}\cdot\left(\tr(B^{(2)})^{k}\right)^{1/k}\cdot\norm{B^{(3)}}_{\spe}.
\]
\end{proof}
\begin{lem}
\label{lem:bound_on_alpha_beta}Under the assumption as Lemma \ref{lem:norm_At}.
Let $u_{t}=u(A_{t})$ with $q=\left\lceil k\right\rceil $ (see Definition
\ref{eq:def_u}), we have that
\[
du_{t}=\alpha_{t}^{T}dW_{t}+\beta_{t}dt
\]
with
\[
\norm{\alpha_{t}}_{2}\apprle u_{t}^{\frac{3}{2}}\quad\text{and}\quad\beta_{t}\apprle ku_{t}^{4}\Phi^{1/k}.
\]
\end{lem}

\begin{proof}
For $\alpha_{t}$, we use Lemma \ref{lem:tensorestimate} and get
that
\begin{align*}
\norm{\E_{x\sim\tilde{p}_{t}}x^{T}(u_{t}I-A_{t})^{-(q+1)}x\cdot x}_{2} & \apprle\norm{A_{t}}_{\spe}^{1/2}\tr(A_{t}(u_{t}I-A_{t})^{-(q+1)})\\
 & \leq u_{t}^{\frac{3}{2}}\kappa_{t}.
\end{align*}

For $\beta_{t}$, we bound each term separately. For the first term,
Lemma \ref{lem:tensor_bound} shows that
\begin{align*}
 & \E_{x,y\sim\tilde{p}_{t}}x^{T}(u_{t}I-A_{t})^{-1}y\cdot x^{T}(u_{t}I-A_{t})^{-(q+1)}y\cdot x^{T}y\\
\apprle & \tr(A_{t}^{\frac{1}{2}}(u_{t}I-A_{t})^{-(q+1)}A_{t}^{\frac{1}{2}})\cdot\left(\tr(A_{t}^{\frac{1}{2}}(u_{t}I-A_{t})^{-1}A_{t}^{\frac{1}{2}})^{k}\right)^{1/k}\cdot\norm{A_{t}}_{\spe}\\
\leq & \norm{A_{t}}_{\spe}^{3}\cdot\kappa_{t}\cdot\left(\tr(u_{t}I-A_{t})^{-k}\right)^{1/k}\\
\leq & u_{t}^{3}\cdot\kappa_{t}\cdot\left(\tr(u_{t}I-A_{t})^{-k}\right)^{1/k}.
\end{align*}
For the second term, we use Lemma \ref{lem:tensorestimate} and get
that
\begin{align*}
 & \E_{x,y\sim\tilde{p}_{t}}x^{T}(u_{t}I-A_{t})^{-(q+1)}x\cdot y^{T}(u_{t}I-A_{t})^{-(q+2)}y\cdot x^{T}y\\
\leq & \norm{\E_{x\sim\tilde{p}_{t}}x^{T}(u_{t}I-A_{t})^{-(q+1)}x\cdot x}_{2}\norm{\E_{y\sim\tilde{p}_{t}}y^{T}(u_{t}I-A_{t})^{-(q+2)}y\cdot y}_{2}\\
\apprle & \norm{A_{t}}_{\spe}^{1/2}\tr(A_{t}(u_{t}I-A_{t})^{-(q+1)})\cdot\norm{A_{t}}_{\spe}^{1/2}\tr(A_{t}(u_{t}I-A_{t})^{-(q+2)})\\
\leq & u_{t}^{3}\cdot\kappa_{t}\cdot\tr((u_{t}I-A_{t})^{-(q+2)}).
\end{align*}
For the third term, the same calculation shows that 
\begin{align*}
 & \tr((u_{t}I-A_{t})^{-(q+2)})\cdot\E_{x,y\sim\tilde{p}_{t}}x^{T}(u_{t}I-A_{t})^{-(q+1)}x\cdot y^{T}(u_{t}I-A_{t})^{-(q+1)}y\cdot x^{T}y\\
\apprle & u_{t}^{3}\cdot\tr((u_{t}I-A_{t})^{-(q+2)})\cdot\kappa_{t}^{2}.
\end{align*}
Combining all three terms, we have 
\begin{align*}
\beta_{t} & \apprle qu_{t}^{3}\left(\left(\tr(u_{t}I-A_{t})^{-k}\right)^{1/k}+\frac{\tr((u_{t}I-A_{t})^{-(q+2)})}{\tr((u_{t}I-A_{t})^{-(q+1)})}\right)\\
 & \apprle qu_{t}^{3}\left(\left(\tr(u_{t}I-A_{t})^{-k}\right)^{1/k}+\tr((u_{t}I-A_{t})^{-(q+2)})^{\frac{1}{q+2}}\right)\\
 & \apprle ku_{t}^{3}\left(\tr(u_{t}I-A_{t})^{-k}\right)^{1/k}
\end{align*}
where we used $q=\left\lceil k\right\rceil $ at the end. Next, we
note that 
\[
\tr(u_{t}I-A_{t})^{-k}\leq u_{t}^{q-k}\tr(u_{t}I-A_{t})^{-q}
\]
and hence $\left(\tr(u_{t}I-A_{t})^{-k}\right)^{1/k}\leq u_{t}\cdot\Phi^{1/k}$.
Therefore, we have the result.
\end{proof}
Now, we are ready to upper bound $\norm{A_{t}}_{\spe}$.
\begin{proof}[Proof of Lemma \ref{lem:norm_At}.]
 Consider the potential $\Psi_{t}=-(u_{t}+1)^{-3}$. Using Lemma
\ref{lem:bound_on_alpha_beta}, we have that
\begin{align*}
d\Psi_{t} & =3(u_{t}+1)^{-4}(\alpha_{t}^{T}dW_{t}+\beta_{t}dt)-6(u_{t}+1)^{-5}\norm{\alpha_{t}}^{2}dt\\
 & \defeq\gamma_{t}^{T}dW_{t}+\eta_{t}dt.
\end{align*}
Note that
\[
\norm{\gamma_{t}}_{2}^{2}=\norm{3(u_{t}+1)^{-4}\alpha_{t}}_{2}^{2}\leq O(1)(u_{t}+1)^{-8}u_{t}^{3}\leq c
\]
and
\[
\eta_{t}\leq3(u_{t}+1)^{-4}O(u_{t}^{4})k\Phi^{1/k}\leq ck\Phi^{1/k}
\]
for some universal constant $c$.

Let $Y_{t}$ be the process $dY_{t}=\gamma_{t}^{T}dW_{t}$. By Theorem
\ref{thm:Dubins}, there exists a Wiener process $\widetilde{W}_{t}$
such that $Y_{t}$ has the same distribution as $\widetilde{W}_{[Y]_{t}}$.
Using the reflection principle for 1-dimensional Brownian motion,
we have that
\[
\P(\max_{t\in[0,T]}Y_{t}\geq\gamma)\leq\P(\max_{t\in[0,cT]}\widetilde{W}_{t}\geq\gamma)=2\P(\widetilde{W}_{cT}\geq\gamma)\leq2\exp(-\frac{\gamma^{2}}{2cT}).
\]
Therefore, we have that
\[
\P(\max_{t\in[0,T]}\Psi_{t}-\Psi_{0}\geq ck\Phi^{1/k}T+\gamma)\leq2\exp(-\frac{\gamma^{2}}{2cT}).
\]
Set $\Phi=2^{-k}n$. At $t=0$, we have $\tr(u_{0}I-I)^{-k}=2^{-k}n$.
Therefore, $u_{0}=\frac{3}{2}$ and $\Psi_{0}=-\frac{8}{125}$. Using
the assumptions that $T\leq\frac{1}{125ck\Phi^{1/k}}$, we have that
\[
\P(\max_{t\in[0,T]}\left(-(u_{t}+1)^{-3}\right)\geq-\frac{1}{125}+\gamma)\leq2\exp(-\frac{\gamma^{2}}{2cT}).
\]
The result follows from setting $\gamma=\frac{1}{100}$.
\end{proof}

\subsection{Calculus with the Stieltjes potential}

Here we prove Lemma \ref{lem:stoc_du}.
\begin{lem}
\label{lem:derivative_u}We have that 
\[
Du(X)[H]=\frac{\tr((uI-X)^{-(q+1)}H)}{\tr((uI-X)^{-(q+1)})}
\]
and
\begin{align*}
D^{2}u(X)[H_{1},H_{2}]= & \frac{\sum_{k=1}^{q+1}\tr((uI-X)^{-k}H_{1}(uI-X)^{-(q+2-k)}H_{2})}{\tr((uI-X)^{-(q+1)})}\\
 & -(q+1)\frac{\tr((uI-X)^{-(q+1)}H_{1})\tr((uI-X)^{-(q+2)}H_{2})}{\tr((uI-X)^{-(q+1)})^{2}}\\
 & -(q+1)\frac{\tr((uI-X)^{-(q+1)}H_{2})\tr((uI-X)^{-(q+2)}H_{1})}{\tr((uI-X)^{-(q+1)})^{2}}\\
 & +(q+1)\frac{\tr((uI-X)^{-(q+1)}H_{1})\tr((uI-X)^{-(q+1)}H_{2})\tr((uI-X)^{-(q+2)})}{\tr((uI-X)^{-(q+1)})^{3}}.
\end{align*}
\end{lem}

\begin{proof}
Taking derivative of the equation (\ref{eq:def_u}), we have 
\[
-q\tr((uI-X)^{-(q+1)})\cdot Du(X)[H]+q\tr((uI-X)^{-(q+1)}H)=0.
\]
Therefore,
\[
Du(X)[H]=\frac{\tr((uI-X)^{-(q+1)}H)}{\tr((uI-X)^{-(q+1)})}.
\]
Taking derivative again on both sides,
\begin{align*}
D^{2}u(X)[H_{1},H_{2}]= & \frac{\sum_{k=1}^{q+1}\tr((uI-X)^{-k}H_{1}(uI-X)^{-(q+2-k)}H_{2})}{\tr((uI-X)^{-(q+1)})}\\
 & -\frac{(q+1)\tr((uI-X)^{-(q+1)}H_{1})\tr((uI-X)^{-(q+2)}H_{2})}{\tr((uI-X)^{-(q+1)})^{2}}\\
 & -(q+1)\frac{\tr((uI-X)^{-(q+2)}H_{1})}{\tr((uI-X)^{-(q+1)})}Du(X)[H_{2}]\\
 & +(q+1)\frac{\tr((uI-X)^{-(q+1)}H_{1})\tr((uI-X)^{-(q+2)})}{\tr((uI-X)^{-(q+1)})^{2}}Du(X)[H_{2}].
\end{align*}
Substituting the formula of $Du(X)[H_{2}]$ and organizing the term,
we have the result.
\end{proof}
To simplify the first term in the Hessian, we need the following Lemma:
\begin{lem}[\cite{Eldan2013}]
\label{lem:tensor_shift}For any positive definite matrix $A$, symmetric
matrix $\Delta$ and any $\alpha,\beta\geq0$, we have that
\[
\tr(A^{\alpha}\Delta A^{\beta}\Delta)\leq\tr(A^{\alpha+\beta}\Delta^{2}).
\]
\end{lem}

\begin{proof}
Without loss of generality, we can assume $A$ is diagonal by rotating
the space. Hence, we have that 
\begin{align*}
\tr(A^{\alpha}\Delta A^{\beta}\Delta) & =\sum_{i,j}A_{ii}^{\alpha}A_{jj}^{\beta}\Delta_{ij}^{2}\\
 & \leq\sum_{i,j}\left(\frac{\alpha}{\alpha+\beta}A_{ii}^{\alpha+\beta}+\frac{\beta}{\alpha+\beta}A_{jj}^{\alpha+\beta}\right)\Delta_{ij}^{2}\\
 & =\frac{\alpha}{\alpha+\beta}\sum_{i,j}A_{ii}^{\alpha+\beta}\Delta_{ij}^{2}+\frac{\beta}{\alpha+\beta}\sum_{i,j}A_{jj}^{\alpha+\beta}\Delta_{ij}^{2}\\
 & =\tr(A^{\alpha+\beta}\Delta^{2}).
\end{align*}
\end{proof}
Now, we are already to upper bound $du(A_{t})$.
\begin{proof}[Proof of Lemma \ref{lem:stoc_du}]
Using Lemma \ref{lem:derivative_u} and Itô's formula, we have that
\begin{align*}
du(A_{t})= & \frac{\tr((uI-A_{t})^{-(q+1)}dA_{t})}{\tr((uI-A_{t})^{-(q+1)})}\\
 & +\frac{1}{2}\sum_{ijkl}\frac{\sum_{k=1}^{q+1}\tr((uI-A_{t})^{-k}e_{ij}(uI-A_{t})^{-(q+2-k)}e_{kl})}{\tr((uI-A_{t})^{-(q+1)})}d[A_{ij},A_{kl}]_{t}\\
 & -\frac{q+1}{2}\sum_{ijkl}\frac{\tr((uI-A_{t})^{-(q+1)}e_{ij})\tr((uI-A_{t})^{-(q+2)}e_{kl})}{\tr((uI-A_{t})^{-(q+1)})^{2}}d[A_{ij},A_{kl}]_{t}\\
 & -\frac{q+1}{2}\sum_{ijkl}\frac{\tr((uI-A_{t})^{-(q+1)}e_{kl})\tr((uI-A_{t})^{-(q+2)}e_{ij})}{\tr((uI-A_{t})^{-(q+1)})^{2}}d[A_{ij},A_{kl}]_{t}\\
 & +\frac{q+1}{2}\sum_{ijkl}\frac{\tr((uI-A_{t})^{-(q+1)}e_{ij})\tr((uI-A_{t})^{-(q+1)}e_{kl})\tr((uI-A_{t})^{-(q+2)})}{\tr((uI-A_{t})^{-(q+1)})^{3}}d[A_{ij},A_{kl}]_{t}.
\end{align*}

For brevity, we let $\tilde{p}_{t}$ be the translation of $p_{t}$
that has mean $0$, i.e. $\tilde{p}_{t}(x)=p_{t}(x+\mu_{t})$. Let
$\Delta^{(z)}=\E_{x\sim\tilde{p}_{t}}xx^{T}x_{z}$. Then, Lemma \ref{lem:dA}
shows that $dA_{t}=\sum_{z}\Delta_{z}dW_{t,z}-A_{t}^{2}dt$ where
$W_{t,z}$ is the $z^{th}$ coordinate of $W_{t}$. Therefore, 
\begin{equation}
d[A_{ij},A_{kl}]_{t}=\sum_{z}\Delta_{ij}^{(z)}\Delta_{kl}^{(z)}dt.\label{eq:dA_bracket}
\end{equation}
Using the formula for $dA_{t}$ (Lemma \ref{lem:dA}) and $d[A_{ij},A_{kl}]_{t}$
(\ref{eq:dA_bracket}), we have that
\begin{align*}
du(A_{t})= & \frac{\tr\left((uI-A_{t})^{-(q+1)}\left(\E_{x\sim\tilde{p}_{t}}xx^{T}x^{T}dW_{t}-A_{t}^{2}dt\right)\right)}{\tr((uI-A_{t})^{-(q+1)})}\\
 & +\frac{1}{2}\sum_{z}\frac{\sum_{k=1}^{q+1}\tr((uI-A_{t})^{-k}\Delta^{(z)}(uI-A_{t})^{-(q+2-k)}\Delta^{(z)})}{\tr((uI-A_{t})^{-(q+1)})}dt\\
 & -(q+1)\sum_{z}\frac{\tr((uI-A_{t})^{-(q+1)}\Delta^{(z)})\tr((uI-A_{t})^{-(q+2)}\Delta^{(z)})}{\tr((uI-A_{t})^{-(q+1)})^{2}}dt\\
 & +\frac{q+1}{2}\sum_{z}\frac{\tr((uI-A_{t})^{-(q+1)}\Delta^{(z)})\tr((uI-A_{t})^{-(q+1)}\Delta^{(z)})\tr((uI-A_{t})^{-(q+2)})}{\tr((uI-A_{t})^{-(q+1)})^{3}}dt.
\end{align*}
Using Lemma \ref{lem:tensor_shift}, 
\[
\tr((uI-A_{t})^{-k}\Delta^{(z)}(uI-A_{t})^{-(q+2-k)}\Delta^{(z)})\leq\tr((uI-A_{t})^{-1}\Delta^{(z)}(uI-A_{t})^{-(q+1)}\Delta^{(z)})
\]
for all $1\leq k\leq q+1$. 

Let $\kappa_{t}=\tr((uI-A_{t})^{-(q+1)})$, then
\begin{align*}
du(A_{t})\leq & \frac{1}{\kappa_{t}}\E_{x\sim\tilde{p}_{t}}x^{T}(uI-A_{t})^{-(q+1)}xx^{T}dW_{t}\\
 & +\frac{q+1}{2\kappa_{t}}\E_{x,y\sim\tilde{p}_{t}}\sum_{z}\tr((uI-A_{t})^{-1}xx^{T}x_{z}(uI-A_{t})^{-(q+1)}yy^{T}y_{z})dt\\
 & -\frac{q+1}{\kappa_{t}^{2}}\E_{x,y\sim\tilde{p}_{t}}\sum_{z}\tr((uI-A_{t})^{-(q+1)}xx^{T}x_{z})\tr((uI-A_{t})^{-(q+2)}yy^{T}y_{z})dt\\
 & +\frac{q+1}{2\kappa_{t}^{3}}\E_{x,y\sim\tilde{p}_{t}}\sum_{z}\tr((uI-A_{t})^{-(q+1)}xx^{T}x_{z})\tr((uI-A_{t})^{-(q+1)}yy^{T}y_{z})\tr((uI-A_{t})^{-(q+2)})dt.
\end{align*}
Rearranging the terms, we have the result.
\end{proof}

\subsection{Bounding the size of any initial set}

Fix any set $E\subset\Rn$ and define $g_{t}=p_{t}(E)$.
\begin{lem}
\label{lem:volume} The random variable $g_{t}$ is a martingale satisfying
\[
d[g_{t}]_{t}\leq D^{2}g_{t}^{2}dt
\]
and
\[
d[g_{t}]_{t}\leq30\norm{A_{t}}_{\spe}\cdot g_{t}^{2}\log^{2}\left(\frac{e}{g_{t}}\right)dt.
\]
\end{lem}

\begin{proof}
Note that
\begin{align*}
dg_{t} & =\left\langle \int_{E}(x-\mu_{t})p_{t}(x)dx,dW_{t}\right\rangle .
\end{align*}
Therefore, we have that
\[
d[g_{t}]_{t}=\norm{\int_{E}(x-\mu_{t})p_{t}(x)dx}_{2}^{2}dt.
\]
We bound the norm in two different ways. On one hand, we note that
\begin{equation}
\norm{\int_{E}(x-\mu_{t})p_{t}(x)dx}_{2}\leq\sup\norm{x-\mu_{t}}_{2}\left(\int_{E}p_{t}(x)dx\right)=\sup\norm{x-\mu_{t}}_{2}g_{t}\leq D\cdot g_{t}\label{eq:gt1}
\end{equation}
On the other hand, for any $\ell\geq1$, we have that 
\begin{align}
\norm{\int_{E}(x-\mu_{t})p_{t}(x)dx}_{2} & =\max_{\norm{\zeta}_{2}=1}\int_{E}(x-\mu_{t})^{T}\zeta\cdot p_{t}(x)dx\nonumber \\
 & \leq\max_{\norm{\zeta}_{2}=1}\left(\int_{E}\left|(x-\mu_{t})^{T}\zeta\right|^{\ell}\cdot p_{t}(x)dx\right)^{\frac{1}{\ell}}\left(\int_{E}p_{t}(x)dx\right)^{1-\frac{1}{\ell}}\nonumber \\
 & \leq2\ell\norm{A_{t}}_{\spe}^{1/2}\cdot g_{t}{}^{1-\frac{1}{\ell}}\label{eq:gt2}
\end{align}
where we used Lemma \ref{lem:lcmom} at the end. Setting $\ell=\log(\frac{e}{g_{t}})$,
we have the result.
\end{proof}
Using this, we can bound how fast $\log\frac{1}{g_{t}}$ changes.
\begin{lem}
\label{lem:set_large} For any $T\geq0$ and $\gamma\geq0$, we have
that
\[
\P\left(\text{ for all }0\leq t\leq T:\,\log\frac{1}{g_{0}}+\frac{1}{2}D^{2}t+\gamma\geq\log\frac{1}{g_{t}}\geq\log\frac{1}{g_{0}}-\gamma\right)\geq1-4\exp(-\frac{\gamma^{2}}{2TD^{2}}).
\]
\end{lem}

\begin{proof}
Since $dg_{t}=g_{t}\alpha_{t}^{T}dW_{t}$ for some $\norm{\alpha_{t}}_{2}\leq D$
(from (\ref{eq:gt1}) in Lemma \ref{lem:volumeKLS}), using Itô's
formula (Lemma (\ref{lem:Ito})) we have that
\begin{align*}
d\log\frac{e}{g_{t}} & =-\frac{dg_{t}}{g_{t}}+\frac{1}{2}\frac{d[g_{t}]_{t}}{g_{t}^{2}}\\
 & =-\alpha_{t}^{T}dW_{t}+\frac{1}{2}\norm{\alpha_{t}}^{2}dt.
\end{align*}
Let $Y_{t}$ be the process $dY_{t}=\alpha_{t}^{T}dW_{t}$. By Theorem
\ref{thm:Dubins} and the reflection principle, we have that
\[
\P(\max_{t\in[0,T]}\left|Y_{t}\right|\geq\gamma)\leq4\exp(-\frac{\gamma^{2}}{2TD^{2}}).
\]
\end{proof}
Now, we bound $\E g_{t}\sqrt{\log\frac{1}{g_{t}}}$. This is the main
result of this section.
\begin{lem}
\label{lem:g_sqrt_g}There is some universal constant $c\geq0$ such
that for any measurable subset $E$ such that $p_{0}(E)\leq\frac{1}{2}$
and any $T$ such that 
\[
0\leq T\leq c\cdot\max\left(\frac{1}{D^{2}}\log\frac{1}{p_{0}(E)},\frac{1}{\log\frac{1}{p_{0}(E)}+D}\right),
\]
we have that
\[
\E\left(p_{T}(E)\sqrt{\log\frac{1}{p_{T}(E)}}1_{p_{T}(E)\leq\frac{1}{2}}\right)\geq\frac{1}{5}p_{0}(E)\sqrt{\log\frac{1}{p_{0}(E)}}.
\]
\end{lem}

\begin{proof}
Fix any set $E\subset\Rn$ and define $g_{t}=p_{t}(E)$. First of
all, if $p_{0}(E)=\Theta(1)$, the statements is true because of Lemma
\ref{lem:set_large}. Hence, we can assume $p_{0}(E)\leq c$ for small
universal constant $c$.

Case 1) $T\leq\frac{1}{8}D^{-2}\log\frac{1}{g_{0}}$. Lemma \ref{lem:set_large}
shows that
\[
\P(\log\frac{1}{g_{t}}\geq\frac{1}{4}\log\frac{1}{g_{0}}\text{ for all }0\leq t\leq T)\geq1-4g_{0}^{2}.
\]
Using $g_{0}\leq\frac{1}{16}$, we have that
\[
\E\left(g_{T}\sqrt{\log\frac{1}{g_{T}}}1_{g_{T}\leq\frac{1}{2}}\right)\geq\E\left(g_{T}\sqrt{\log\frac{1}{g_{T}}}1_{\log\frac{1}{g_{T}}\geq\frac{1}{4}\log\frac{1}{g_{0}}}\right)\geq\E\left(g_{T}1_{\log\frac{1}{g_{T}}\geq\frac{1}{4}\log\frac{1}{g_{0}}}\right)\sqrt{\frac{1}{4}\log\frac{1}{g_{0}}}.
\]
Since $\E g_{T}=g_{0}$ and $g_{T}\leq1$, we have that 
\[
\E\left(g_{T}1_{\log\frac{1}{g_{T}}\geq\frac{1}{4}\log\frac{1}{g_{0}}}\right)=g_{0}-\E\left(g_{T}1_{\log\frac{1}{g_{T}}<\frac{1}{4}\log\frac{1}{g_{0}}}\right)\geq g_{0}-4g_{0}^{2}\geq\frac{1}{2}g_{0}.
\]
Therefore, 
\[
\E\left(g_{T}\sqrt{\log\frac{1}{g_{T}}}1_{g_{T}\leq\frac{1}{2}}\right)\geq\frac{1}{4}g_{0}\sqrt{\log\frac{1}{g_{0}}}.
\]

Case 2) $T\geq\frac{1}{8}D^{-2}\log\frac{1}{g_{0}}$. Now, we assume
that $T\leq\frac{1}{2c(D+\log\frac{1}{g_{0}})}$ where $c\geq1$ is
the universal constant appears in Lemma \ref{lem:norm_At}. Note that
\begin{align*}
dg_{t}\sqrt{\log\frac{e}{g_{t}}} & =\frac{2\log\frac{e}{g_{t}}-1}{2\sqrt{\log\frac{e}{g_{t}}}}dg_{t}-\frac{2\log\frac{e}{g_{t}}+1}{8g_{t}\log^{\frac{3}{2}}\frac{e}{g_{t}}}d[g_{t}]_{t}.
\end{align*}
Since $dg_{t}=g_{t}\log\frac{e}{g_{t}}\alpha_{t}^{T}dW_{t}$ for some
$\norm{\alpha_{t}}_{2}\leq\sqrt{30}\norm{A_{t}}_{\spe}^{1/2}$ (from
(\ref{eq:gt2}) in Lemma \ref{lem:volumeKLS}), 
\begin{align*}
dg_{t}\sqrt{\log\frac{e}{g_{t}}} & =\frac{1}{2}g_{t}\sqrt{\log\frac{e}{g_{t}}}(2\log\frac{e}{g_{t}}-1)\alpha_{t}^{T}dW_{t}-\frac{1}{8}g_{t}\sqrt{\log\frac{e}{g_{t}}}(2\log\frac{e}{g_{t}}+1)\norm{\alpha_{t}}_{2}^{2}dt.
\end{align*}
For any $s\geq s'\geq0$, we have that
\begin{align}
\E g_{s}\sqrt{\log\frac{e}{g_{s}}} & =g_{s'}\sqrt{\log\frac{e}{g_{s'}}}-\frac{1}{8}\int_{s'}^{s}\E\left(g_{t}\sqrt{\log\frac{e}{g_{t}}}(2\log\frac{e}{g_{t}}+1)\norm{\alpha_{t}}_{2}^{2}\right)dt\nonumber \\
 & \geq g_{s'}\sqrt{\log\frac{e}{g_{s'}}}-12\int_{s'}^{s}\E\left(\norm{A_{t}}_{\spe}g_{t}\log^{\frac{3}{2}}\frac{e}{g_{t}}\right)dt.\label{eq:glogg1}
\end{align}

Using $T\geq\frac{1}{8}D^{-2}\log\frac{1}{g_{0}}$, Lemma \ref{lem:set_large}
shows that
\begin{equation}
\P(15D^{2}T\geq\max_{0\leq t\leq T}\log\frac{1}{g_{t}})\geq1-4g_{0}^{2}.\label{eq:glogg2}
\end{equation}
Now, using $T\leq\frac{1}{2c(\sqrt{n}+\log\frac{1}{g_{0}})}$, Lemma
\ref{lem:norm_At} (with $k=2$) shows that
\begin{equation}
\P(\max_{t\in[0,T]}\norm{A_{t}}_{\spe}\geq2)\leq2g_{0}^{2}.\label{eq:glogg3}
\end{equation}
Let $E$ be the event that both $\max_{0\leq t\leq T}\log\frac{1}{g_{t}}\leq15D^{2}T$
and $\max_{t\in[0,T]}\norm{A_{t}}_{\spe}\leq2$. Then, combining (\ref{eq:glogg2})
and (\ref{eq:glogg3}), we have that
\begin{align*}
\E\left(\norm{A_{t}}_{\spe}g_{t}\log^{\frac{3}{2}}\frac{e}{g_{t}}\right) & \leq2(1+15D^{2}T)\E\left(g_{t}\log^{\frac{1}{2}}\frac{e}{g_{t}}1_{E}\right)+\E\left(\norm{A_{t}}g_{t}\log^{\frac{3}{2}}\frac{e}{g_{t}}1_{E^{c}}\right)\\
 & \leq2(1+15D^{2}T)\E\left(g_{t}\log^{\frac{1}{2}}\frac{e}{g_{t}}\right)+12\sqrt{D}g_{0}^{2}.
\end{align*}
where we used that $\norm{A_{t}}_{\spe}\leq\sqrt{D}$ a.s. and $g_{t}\log^{\frac{3}{2}}\frac{e}{g_{t}}\leq2$
and $\P(E^{c})\leq6g_{0}^{2}$. Putting it into (\ref{eq:glogg1}),
for any $s'\leq s\leq T$, we have that 
\begin{align*}
\E g_{s}\sqrt{\log\frac{e}{g_{s}}} & \geq g_{s'}\sqrt{\log\frac{e}{g_{s'}}}-24(1+15D^{2}T)\int_{s'}^{s}\E g_{t}\sqrt{\log\frac{e}{g_{t}}}dt-144\sqrt{D}g_{0}^{2}T.
\end{align*}
Let $s^{*}$ be the $s\in[0,T]$ that maximizes $\E g_{s}\sqrt{\log\frac{e}{g_{s}}}$
and $f_{t}=\E g_{t}\sqrt{\log\frac{e}{g_{t}}}$. For all $T\geq s\geq s^{*}$,
we have that
\begin{align*}
f_{s} & \geq f_{s^{*}}-24(1+15D^{2}T)\int_{s'}^{s}f_{t}dt-144\sqrt{D}g_{0}^{2}T\\
 & \geq f_{s^{*}}-24(T+15D^{2}T^{2})f_{s^{*}}-144\sqrt{D}g_{0}^{2}T\\
 & \geq\frac{4}{5}f_{s^{*}}\geq\frac{4}{5}f_{0}
\end{align*}
where we used that $T\leq\frac{1}{10^{5}D}\leq\frac{1}{10^{5}}$ and
$f_{s^{*}}\geq f_{0}\geq g_{0}$ at the end. Therefore, we have
\[
\E g_{T}\sqrt{\log\frac{e}{g_{T}}}\geq\frac{4}{5}g_{0}\sqrt{\log\frac{e}{g_{0}}}.
\]
Now, we note that
\begin{align*}
\E\left(g_{T}\sqrt{\log\frac{1}{g_{T}}}1_{g_{T}\leq\frac{1}{2}}\right) & \geq\frac{1}{2}\E\left(g_{T}\sqrt{\log\frac{e}{g_{T}}}1_{g_{T}\leq\frac{1}{2}}\right)\geq\frac{2}{5}g_{0}\sqrt{\log\frac{e}{g_{0}}}-\frac{\sqrt{\log2e}}{2}g_{0}\\
 & \geq\frac{1}{5}g_{0}\sqrt{\log\frac{e}{g_{0}}}
\end{align*}
where we used $g_{0}\leq\frac{1}{e^{10}}$ at the end.
\end{proof}

\subsection{Gaussian Case}

The next theorem can be found in \cite[Theorem 1.1]{Ledoux1999}.
We give a proof here for completeness.
\begin{thm}
\label{thm:Gaussian-iso-2}Let $h(x)=f(x)e^{-\frac{t}{2}\norm x^{2}}/\int f(y)e^{-\frac{t}{2}\norm y^{2}}dy$
where $f:\R^{n}\rightarrow\R_{+}$ is an integrable logconcave function
and $t\geq0$. Let $p(S)=\int_{S}h(x)dx$. For any $p(S)\leq\frac{1}{2}$,
we have
\[
p(\partial S)=\Omega\left(\sqrt{t}\right)\cdot p(S)\sqrt{\log\frac{1}{p(S)}}.
\]
\end{thm}

\begin{proof}
Let $g=p(S)$. Then, the desired statement can be written as 
\[
\int_{S}h(x)\:dx=g\int_{\R^{n}}h(x)\implies\int_{\partial S}h(x)\:dx\ge c\sqrt{t}g\sqrt{\log(\frac{1}{g})}\int_{\R^{n}}h(x)
\]
for some constant $c$. By the localization lemma \cite{KLS95}, if
there is a counterexample, there is a counterexample in one-dimension
where $h$ is of the form $h(x)=Ce^{-\gamma x-t\frac{x^{2}}{2}}$
restricted to some interval on the real line. Without loss of generality,
we can assume that $S$ is a single interval, otherwise, any interval
has smaller boundary measure and smaller volume. By rescaling, flipping
and shifting the function $h$, we can assume that $t=1$ and that
$h(x)=e^{-\frac{1}{2}x^{2}}1_{[a,b]}$ and that $S=[y,b]$ for some
$a<y<b$. 

It remains to show that for $g\le\frac{1}{2}$, 
\[
\frac{\int_{y}^{b}e^{-\frac{x^{2}}{2}}\,dx}{\int_{a}^{b}e^{-\frac{x^{2}}{2}}\:dx}=g\implies\frac{e^{-\frac{y^{2}}{2}}}{\int_{a}^{b}e^{-\frac{x^{2}}{2}}\,dx}\apprge g\sqrt{\log\frac{1}{g}}.
\]

There are three cases: $y\geq1$, $y\leq-1$ and $-1\leq y\leq-1$.
Let $A=\int_{a}^{b}e^{-\frac{x^{2}}{2}}\:dx$. In the first case $y\geq1$,
we note that the integral 
\begin{equation}
g\cdot A=\int_{y}^{b}e^{-\frac{x^{2}}{2}}\,dx\leq\int_{y}^{\infty}e^{-\frac{x^{2}}{2}}\,dx\apprle e^{-\frac{y^{2}}{2}}/y.\label{eq:gaussian_case_1}
\end{equation}
Rearrange the terms, we have that
\[
\frac{e^{-\frac{y^{2}}{2}}}{\int_{a}^{b}e^{-\frac{x^{2}}{2}}\,dx}\apprge g\cdot y.
\]
Using (\ref{eq:gaussian_case_1}), we have that $y\apprge\sqrt{\log\frac{1}{g\cdot A}}$
and hence
\[
\frac{e^{-\frac{y^{2}}{2}}}{\int_{a}^{b}e^{-\frac{x^{2}}{2}}\,dx}\apprge g\cdot\sqrt{\log\frac{1}{g\cdot A}}\apprge g\cdot\sqrt{\log\frac{1}{g}}
\]
where we used that $A=\int_{a}^{b}e^{-\frac{x^{2}}{2}}\:dx\leq\sqrt{2\pi}$
and $g\leq\frac{1}{2}$.

For the second case $y\leq-1$, since $g\leq\frac{1}{2}$, we have
that
\[
\frac{e^{-\frac{y^{2}}{2}}}{\int_{a}^{b}e^{-\frac{x^{2}}{2}}\,dx}\geq\frac{e^{-\frac{y^{2}}{2}}}{2\int_{a}^{y}e^{-\frac{x^{2}}{2}}\,dx}\geq\frac{e^{-\frac{y^{2}}{2}}}{2\int_{-\infty}^{y}e^{-\frac{x^{2}}{2}}\,dx}\apprge\left|y\right|\geq1
\]
where we used that $\int_{-\infty}^{y}e^{-\frac{x^{2}}{2}}\,dx=\frac{O(1)}{\left|y\right|}e^{-\frac{y^{2}}{2}}$
at the end. This proves the second case because $g\sqrt{\log\frac{1}{g}}\apprle1$.

For the last case $|y|\leq1$, we note that
\[
\frac{e^{-\frac{y^{2}}{2}}}{\int_{a}^{b}e^{-\frac{x^{2}}{2}}\,dx}\geq\frac{e^{-\frac{1}{2}}}{\sqrt{2\pi}}\apprge1.
\]
\end{proof}

\subsection{Proof of the log-Cheeger bound}

We can now prove a bound on the isoperimetric constant. 
\begin{proof}[Proof of Theorem. \ref{lem:expand_p}.]
 By Lemma \ref{lem:def-pt}, $p_{t}$ is a martingale and therefore
\[
p(\partial E)=p_{0}(\partial E)=\E p_{T}(\partial E).
\]
Next, by the definition of $p_{T}$ (\ref{eq:dBt}), we have that
$p_{T}(x)\propto e^{c_{T}^{T}x-\frac{T}{2}\norm x^{2}}p(x)$ and Theorem
\ref{thm:Gaussian-iso-2} shows that if $p_{t}(E)\leq\frac{1}{2}$,
we have that
\[
p_{T}(\partial E)\apprge\sqrt{T}\cdot p_{t}(E)\sqrt{\log\frac{1}{p_{t}(E)}}.
\]
Hence, we have that
\[
p(\partial E)\apprge\sqrt{T}\cdot\E\left(p_{t}(E)\sqrt{\log\frac{1}{p_{t}(E)}}\cdot1_{p_{t}(E)\leq\frac{1}{2}}\right).
\]

Lemma \ref{lem:g_sqrt_g} shows that if 
\[
T\leq c\cdot\max\left(\frac{1}{D^{2}}\log\frac{1}{p_{0}(E)},\frac{1}{\log\frac{1}{p_{0}(E)}+D}\right),
\]
we have
\begin{align*}
p(\partial E) & \apprge\sqrt{T}\cdot p_{0}(E)\sqrt{\log\frac{1}{p_{0}(E)}}\\
 & \apprge\left(\frac{\sqrt{\log\frac{1}{p_{0}(E)}}}{D}+\frac{1}{\sqrt{\log\frac{1}{p_{0}(E)}+D}}\right)p_{0}(E)\sqrt{\log\frac{1}{p_{0}(E)}}\\
 & \apprge\left(\frac{\log\frac{1}{p_{0}(E)}}{D}+\sqrt{\frac{\log\frac{1}{p_{0}(E)}}{\log\frac{1}{p_{0}(E)}+D}}\right)p_{0}(E)\\
 & \apprge\left(\frac{\log\frac{1}{p_{0}(E)}}{D}+\sqrt{\frac{\log\frac{1}{p_{0}(E)}}{D}}\right)p_{0}(E).
\end{align*}
\end{proof}

\section{Consequences of the improved isoperimetric inequality}

In this section, we give some consequences of the improved isoperimetric
inequality (Theorem \ref{lem:expand_p}). 

\subsection{Log-Cheeger inequality}

First we note the following Cheeger-type logarithmic isoperimetric
inequality.
\begin{thm}[\cite{ledoux1994simple}]
Let $\mu$ be any absolutely continuous measure on $\Rn$ such that
for any open subsets $A$ of $\Rn$ with $\mu(A)\leq\frac{1}{2}$,
\[
\mu(\partial A)\geq\phi\cdot\mu(A)\sqrt{\log\frac{1}{\mu(A)}}.
\]
Then, for any $f$ such that $\int_{\Rn}f^{2}d\mu=1$, we have that
\[
\int_{\Rn}\left|\nabla f\right|^{2}d\mu\gtrsim\phi^{2}\cdot\int_{\Rn}f^{2}\log f^{2}d\mu.
\]
\end{thm}

Applying this and Theorem \ref{lem:expand_p}, we have the following
result.
\begin{thm}
Given an isotropic logconcave distribution $p$ with diameter $D$.
For any $f$ such that $\int_{\Rn}f^{2}(x)p(x)dx=1$, we have that
\[
\int_{\Rn}\left|\nabla f(x)\right|^{2}p(x)dx\gtrsim\frac{1}{D}\cdot\int_{\Rn}f^{2}\log f^{2}p(x)dx.
\]
\end{thm}

\begin{rem*}
In general, isotropic logconcave distributions with diameter $D$
can have the coefficient as small as $O(\frac{1}{D})$, as we show
in Section \ref{sec:tight}. Therefore, our bound is tight up to constant.
\end{rem*}
Now we prove Theorem \ref{thm:conc}, an improved concentration inequality
for Lipschitz functions on isotropic logconcave densities.
\begin{proof}[Proof of Theorem \ref{thm:conc}.]
 We first prove the statement for the function $g(x)=\norm x$. Define
\[
E_{t}=\left\{ x\text{ such that }\norm x\geq\text{med}_{x\sim p}\norm x+t\right\} 
\]
 and $\alpha_{t}\defeq\log\frac{1}{p(E_{t})}$. By the definition
of median, we have that $\alpha_{0}=\log2$. We first give a weak
estimate on how fast $\alpha_{t}$ increases. Note that 
\[
\frac{d\alpha_{t}}{dt}=-\frac{1}{p(E_{t})}\frac{dp(E_{t})}{dt}=\frac{p(\partial E_{t})}{p(E_{t})}\geq\frac{c_{1}}{n^{1/4}}
\]
for some universal constant $c_{1}>0$ where we used that definition
of $p(\partial E_{t})$ to get $\frac{dp(E_{t})}{dt}=-p(\partial E_{t})$
and Theorem \ref{thm:n14} at the end. Therefore, 
\begin{equation}
\alpha_{t+s}\geq\alpha_{t}+c_{1}\cdot sn^{-1/4}\label{eq:weak_bound}
\end{equation}
for all $t,s\geq0$. To improve on this bound, we consider the distribution
$p_{t}$ defined by truncating the distribution $p$ to the set $\norm x\le\text{med}_{x\sim p}\norm x+t+c_{2}\sqrt{n}$
for some large enough constant $c_{2}$. By the estimate (\ref{eq:weak_bound}),
we see that $p(E_{t})$ only decreases by a tiny factor after truncation
and hence
\[
\frac{p(\partial E_{t})}{p(E_{t})}\geq\frac{1}{2}\frac{p_{t}(\partial E_{t})}{p_{t}(E_{t})}.
\]
Next, we note that $p_{t}$ is almost isotropic, namely its covariance
matrix $A_{t}$ satisfies $\frac{1}{2}I\preceq A_{t}\preceq2I$ (in
fact, it is exponentially close to $I$). Although Theorem \ref{lem:expand_p}
only applies to the isotropic case, but we can always renormalize
the distribution $p_{t}$ to isotropic and then normalize it back.
Since the distribution $p_{t}$ is almost isotropic, such a re-normalization
does not change $p_{t}(\partial E)$ by more than a multiplicative
constant. 

Therefore, we can apply Theorem \ref{lem:expand_p} and get that
\begin{equation}
\frac{d\alpha_{t}}{dt}=\frac{p(\partial E_{t})}{p(E_{t})}\geq\frac{1}{2}\frac{p_{t}(\partial E_{t})}{p_{t}(E_{t})}\gtrsim\sqrt{\frac{\log\frac{1}{p_{t}(E_{t})}}{m+t}}\geq c_{3}\cdot\sqrt{\frac{\alpha_{t}}{m+t}}\label{eq:dalpha}
\end{equation}
for some universal constant $c_{3}>0$ where $m=\text{med}_{x\sim p}\norm x+c_{2}\sqrt{n}$.
Note that
\[
\frac{d\sqrt{\alpha_{t}}}{dt}\geq\frac{c_{3}}{2}\cdot\sqrt{\frac{1}{m+t}}
\]
and hence
\[
\sqrt{\alpha_{t}}-\sqrt{\alpha_{0}}\geq\frac{c_{3}}{2}\int_{0}^{t}\sqrt{\frac{1}{m+s}}ds=c_{3}\cdot\left(\sqrt{m+t}-\sqrt{m}\right).
\]

For $t\leq m$, we have that
\[
\sqrt{\alpha_{t}}-\sqrt{\alpha_{0}}\geq c_{3}\cdot\left(\sqrt{m}+\sqrt{m}\frac{t}{3m}-\sqrt{m}\right)\geq\frac{c_{3}\cdot t}{3\sqrt{m}}.
\]
and hence $\alpha_{t}\geq\frac{c_{3}^{2}}{9m}t^{2}.$

For $t\geq m$, we note that $\sqrt{1+x}-1\geq\frac{1}{3}\sqrt{x}$
for all $x\geq1$. Therefore, 
\[
\sqrt{m+t}-\sqrt{m}=\sqrt{m}\left(\sqrt{1+\frac{t}{m}}-1\right)\geq\frac{1}{3}\sqrt{m}\sqrt{\frac{t}{m}}=\frac{1}{3}\sqrt{t}.
\]
Hence, (\ref{eq:dalpha}) shows that $\sqrt{\alpha_{t}}-\sqrt{\alpha_{0}}\geq\frac{c_{3}}{3}\sqrt{t}.$

Combining both cases, we have that $\alpha_{t}\gtrsim\min\left(\frac{t^{2}}{\sqrt{n}},t\right)\gtrsim\frac{t^{2}}{t+\sqrt{n}}.$
Hence,
\begin{equation}
\P_{x\sim p}\left(\norm x-\text{med}_{x\sim p}\norm x\geq t\right)\leq\exp\left(-\Omega(1)\cdot\frac{t^{2}}{t+\sqrt{n}}\right).\label{eq:norm_x_small}
\end{equation}
By the same proof, we have that
\[
\P_{x\sim p}\left(\norm x-\text{med}_{x\sim p}\norm x\leq-t\right)\leq\exp\left(-\Omega(1)\cdot\frac{t^{2}}{t+\sqrt{n}}\right).
\]
This completes the proof for the Euclidean norm. For a general $1$-Lipschitz
function $g$, we define 
\[
E_{t}=\left\{ x\text{ such that }g(x)\geq m_{g}+t\right\} 
\]
where $m_{g}=\text{med}_{x\sim p}g(x)$. Again, we define $\alpha_{t}\defeq\log\frac{1}{p(E_{t})}$
and we have $\alpha_{0}=\log2$. To compute $\frac{p(\partial E_{t})}{p(E_{t})}$,
we consider the restriction of $p$ to a large ball. Let $\zeta\geq0$
to be chosen later and $B_{t,\zeta}$ be the ball centered at $0$
with radius $c_{4}\cdot(\sqrt{n}+\alpha_{t})+\zeta$ where $c_{4}$
is a constant. Let $p_{t,\zeta}$ be the distribution defined by 
\[
p_{t,\zeta}(A)=\frac{p(A\cap B_{t,\zeta})}{p(B_{t,\zeta})}.
\]
Choosing a large constant $c_{4}$ and using (\ref{eq:norm_x_small}),
we have that $p(B_{t,\zeta}^{c})\leq\frac{p(E_{t})}{100n}$ and that
$p_{t,\zeta}$ is almost isotropic. Since $p(B_{t,\zeta}^{c})$ is
so small, we have that $p_{t,\zeta}(E_{t})\geq\frac{1}{2}p(E_{t})$
and that
\[
p_{t,\zeta}(\partial E_{t})=\frac{p(\partial(E_{t}\cap B_{t,\zeta}))}{p(B_{t,\zeta})}\leq2p(\partial E_{t})+2p(\partial B_{t,\zeta}).
\]
Hence, 
\begin{equation}
\frac{p(\partial E_{t})}{p(E_{t})}\geq\frac{1}{4}\frac{p_{t,\zeta}(\partial E_{t})}{p_{t,\zeta}(E_{t})}-\frac{p_{t,\zeta}(\partial B_{t,\zeta})}{p(E_{t})}.\label{eq:pE_ratio}
\end{equation}
Since $p_{t,\zeta}$ is almost isotropic and has support of diameter
$O(\sqrt{n}+\alpha_{t}+\zeta)$, Theorem \ref{lem:expand_p} gives
that 
\begin{equation}
\frac{p_{t,\zeta}(\partial E_{t})}{p_{t,\zeta}(E_{t})}\gtrsim\sqrt{\frac{\log\frac{1}{p_{t,\zeta}(E_{t})}}{\sqrt{n}+\alpha_{t}+\zeta}}\geq c_{5}\sqrt{\frac{\alpha_{t}}{\sqrt{n}+\alpha_{t}+\zeta}}\label{eq:ptE_ratio}
\end{equation}
for some universal constant $0<c_{5}<1$. To bound the second term
$p_{t,\zeta}(\partial B_{t,\zeta})$, we note that
\begin{align*}
\int_{0}^{1/c_{5}^{2}}p_{t,\zeta}(\partial B_{t,\zeta})d\zeta & \leq2\int_{0}^{1/c_{5}^{2}}p(\partial B_{t,\zeta})d\zeta\leq2p(B_{t,0}^{c})\leq\frac{p(E_{t})}{50n}
\end{align*}
Hence, there is $\zeta$ between $0$ and $1/c_{5}^{2}$ such that
$p_{t,\zeta}(\partial B_{t,\zeta})\leq\frac{p(E_{t})}{50n}c_{5}^{2}$.

Using this, (\ref{eq:pE_ratio}) and (\ref{eq:ptE_ratio}), we have
that
\begin{align}
\frac{p(\partial E_{t})}{p(E_{t})} & \geq\frac{c_{5}}{4}\sqrt{\frac{\alpha_{t}}{\sqrt{n}+\alpha_{t}+1/c_{5}^{2}}}-\frac{c_{5}^{2}}{50n}\nonumber \\
 & \geq\frac{c_{5}}{8}\sqrt{\frac{\alpha_{t}}{\sqrt{n}+\alpha_{t}+1/c_{5}^{2}}}+\frac{c_{5}}{8}\sqrt{\frac{\log2}{\sqrt{n}+\log2+1/c_{5}^{2}}}-\frac{c_{5}^{2}}{50n}\nonumber \\
 & \geq\frac{c_{5}}{8}\sqrt{\frac{\alpha_{t}}{\sqrt{n}+\alpha_{t}+1/c_{5}^{2}}}.\label{eq:p_ratio_est}
\end{align}

Next, we relate $\frac{p(\partial E_{t})}{p(E_{t})}$ with $dp(E_{t})$.
Since $g$ is $1$-Lipschitz, for any $x$ such that $\norm{x-y}_{2}\leq h$
and $y\in E_{t}$, we have that
\[
g(x)\geq m_{g}+t-h
\]
 Therefore, we have $B(E_{t},h)\subset E_{t-h}$ and, 
\begin{align*}
-\frac{dp(E_{t})}{dt} & =\lim_{h\rightarrow0}\frac{p\left(\left\{ g(x)\geq\text{med}_{x\sim p}g(x)+t-h\right\} \right)-p(E_{t})}{h}\\
 & \geq\lim_{h\rightarrow0}\frac{p\left(B(E_{t},h)\right)-p(E_{t})}{h}=p(\partial E_{t}).
\end{align*}
Using this with (\ref{eq:p_ratio_est}), we have that
\begin{align*}
\frac{d\alpha_{t}}{dt} & =-\frac{1}{p(E_{t})}\frac{dp(E_{t})}{dt}\geq\frac{p(\partial E_{t})}{p(E_{t})}\geq\frac{c_{5}}{8}\sqrt{\frac{\alpha_{t}}{\sqrt{n}+\alpha_{t}+1/c_{5}^{2}}}.
\end{align*}
Solving this equation, we again have that $\alpha_{t}\gtrsim\frac{t^{2}}{t+\sqrt{n}}.$
This proves that 
\[
\P_{x\sim p}\left(g(x)-\text{med}_{x\sim p}g(x)\geq t\right)\leq\exp\left(-\Omega(1)\cdot\frac{t^{2}}{t+\sqrt{n}}\right).
\]

The case of $g(x)-\text{med}_{x\sim p}g(x)\leq-t$ is the same by
taking the negative of $g$. To replace $\text{med}_{x\sim p}g(x)$
by $\E_{x\sim p}g(x)$, we simply use the concentration we just proved
to show that $\left|\E_{x\sim p}g(x)-\text{med}_{x\sim p}g(x)\right|\lesssim n^{\frac{1}{4}}$.
\end{proof}

\subsection{Small ball probability}

To bound the small ball probability, we use the following theorem
that gives the small ball estimate for logconcave distributions with
a subgaussian tail, including strongly logconcave distributions.
\begin{thm}[Theorem 1.3 in \cite{paouris2012small}]
\label{thm:smallpaouris}Given a logconcave distribution $p$ with
covariance matrix $A$. Suppose that $p$ has subgaussian constant
$b$, i.e.
\[
\E_{x\sim p}\left\langle x,\theta\right\rangle ^{2k}\leq\beta^{2k}\E_{g\sim\gamma}g^{2k}\text{ for all }\theta\in\Rn,k=1,2,\cdots
\]
where $\gamma$ is a standard Gaussian random variable. For any $y\in\Rn$
and $0\leq\varepsilon\leq c_{1}$, we have that
\[
\P_{x\sim p}\left(\norm{x-y}_{2}^{2}\leq\varepsilon\tr A\right)\leq\varepsilon^{c_{2}\beta^{-2}\norm A_{\spe}^{-1}\tr A}
\]
for some positive universal constant $c_{1}$ and $c_{2}$.
\end{thm}

We can now use stochastic localization to bound the small ball probability
(Theorem \ref{thm:small-ball}).
\begin{proof}[Proof of Theorem \ref{thm:small-ball}.]
Lemma \ref{lem:norm_At} shows that
\begin{equation}
\norm{A_{t}}_{\spe}\leq2\text{ for all }0\leq t\leq\frac{c}{kn^{1/k}}\label{eq:event_A}
\end{equation}
with probability at least $\frac{9}{10}$ for some universal constant
$c>0$.

Now, we use it to bound $\tr A_{t}$. Lemma \ref{lem:dA} shows that
\[
d\tr A_{t}=\E_{x}\norm{x-\mu_{t}}^{2}(x-\mu_{t})^{T}dW_{t}-\tr A_{t}^{2}.
\]
Under the event \ref{eq:event_A}, we have that $\tr A_{t}^{2}\leq4n$.
Using Lemma \ref{lem:tensorestimate}, we have that
\[
\norm{\E_{x}\norm{x-\mu_{t}}^{2}(x-\mu_{t})}_{2}\lesssim n.
\]
Therefore, $d\tr A_{t}\geq O(n)dW_{t}-O(n)$. Hence, $\tr A_{t}$
is lower bounded by a Gaussian with mean $(1-c't)n$ and variance
$c''n^{2}t$ for some constants $c'$ and $c''$. Using this we have
that
\begin{equation}
\tr A_{t}\gtrsim n\text{ for all }0\leq t\leq c''\label{eq:event_B}
\end{equation}
with probability at least $\frac{8}{10}$ for some universal constant
$c''$.

Next, we let $E=\{x:\norm x^{2}\leq\varepsilon n\}$. We can assume
the diameter of $p$ is $2\sqrt{n}$ by truncating the domain, it
does not affect the small ball probability by more than a constant
factor. Let $g_{t}=p_{t}(E)$. Lemma \ref{lem:volume} shows that
\[
d[g_{t}]_{t}\leq(2\sqrt{n})^{2}g_{t}^{2}dt.
\]
Using this, we have that 
\[
d\log g_{t}=\eta_{t}dW_{t}-\frac{1}{2}\frac{1}{g_{t}^{2}}d[g_{t}]_{t}\geq\eta_{t}dW_{t}-(2\sqrt{n})^{2}dt
\]
with $0\leq\eta_{t}\leq2\sqrt{n}$. Solving this equation, with probability
at least $\frac{9}{10}$, we have 
\begin{equation}
g_{t}\geq e^{-O(nt)}g_{0}.\label{eq:event_C}
\end{equation}

For the distribution $p_{t}$, we note that $p_{t}(x)=e^{-\frac{t}{2}\norm x_{2}^{2}}q_{t}(x)$
for some logconcave distribution $q_{t}$ (Definition \ref{def:A}).
Therefore, $p_{t}$ has subgaussian constant at most $O(\frac{1}{\sqrt{t}})$.
Theorem \ref{thm:smallpaouris} shows that
\[
\P_{x\sim p_{t}}\left(\norm{x-y}_{2}^{2}\leq\varepsilon\tr A_{t}\right)\leq\varepsilon^{\Omega(t\norm{A_{t}}_{\spe}^{-1}\tr A_{t})}.
\]
Therefore, under the events (\ref{eq:event_A}) and (\ref{eq:event_B}),
we have that 
\[
g_{t}\leq\exp(-\Omega(nt\log\frac{1}{\varepsilon})).
\]
Since the events (\ref{eq:event_A}), (\ref{eq:event_B}) and \ref{eq:event_C}
each happens with probability $\frac{9}{10}$, we have that with probability
at least $\frac{7}{10}$, with $t=\Omega(k^{-1}n^{-1/k})$, 
\[
g_{0}\leq\exp(O(nt)-\Omega(nt\log\frac{1}{\varepsilon}))\leq\exp(-\Omega(nt\log\frac{1}{\varepsilon})).
\]
This completes the proof.
\end{proof}

\section{\label{sec:speedy}Convergence of the speedy walk}

Here we bound the log-Cheeger constant of the speedy walk Markov chain.
This theorem and Theorem 1.2 from \cite{KannanLM06} implies the mixing
time of the speedy walk (Theorem \ref{thm:speedy}).
\begin{thm}
\label{thm:speedy-log-Cheeger}For the speedy walk with $\delta\lesssim1$
steps applied to an isotropic convex body $K$ with diameter $D$,
the expansion $\psi(s)$ of any subset of measure $e^{-O(D/\delta^{2})}\leq s\le\frac{1}{2}$
satisfies

\[
\psi(s)\apprge\frac{\delta}{\sqrt{n}}\left(\sqrt{\frac{\log\frac{1}{s}}{D}}+\frac{\log\frac{1}{s}}{D}\right).
\]
\end{thm}

Recall the local conductance is $\ell(x)\defeq\frac{\vol(K\cap(x+\delta B_{n}))}{\vol(\delta B_{n})}$.
We will need the following facts about it.
\begin{lem}
\label{lem:local}For any convex body $K$, the function $\ell(x)$
is logconcave and the stationary distribution of the speedy walk is
proportional to it. Moreover, assuming $K$ contains a unit ball,
the subset 

\[
K'=\left\{ x\in K\,:\,\ell(x)\ge\frac{3}{4}\right\} 
\]
satisfies $\vol(K')\ge(1-2\delta\sqrt{n})\vol(K)$.
\end{lem}

Next we need a small refinement of Theorem \ref{thm:Gaussian-iso-2}
about the isoperimetry of a density proportional to a Gaussian times
a logconcave function. For two points $u,v$ in the support of a nonnegative
function $h$, let 
\[
d_{h}(u,v)=\frac{\left|h(u)-h(v)\right|}{\max\left\{ h(u),h(v)\right\} }.
\]

\begin{thm}
\label{thm:gaussian-speedy}Let $h(x)=f(x)e^{-\frac{t}{2}\norm x_{2}^{2}}/\int f(y)e^{-\frac{t}{2}\norm y_{2}^{2}}dy$
where $f:\R^{n}\rightarrow\R_{+}$ is an integrable logconcave function
and $t\geq0$. Let $p(S)=\int_{S}h(x)dx$. Let $S_{1},S_{2},S_{3}$
partition of $\R^{n}$ such that $p(S_{1})\le p(S_{2})$ and for any
$u\in S_{1},v\in S_{2}$, either $\|u-v\|_{2}\ge d$ or $d_{h}(u,v)\ge d_{h}$.
Then, 
\[
p(S_{3})\apprge\min\left\{ \frac{d_{h}}{\sqrt{n}},d\sqrt{t\log\frac{1}{p(S_{1})}}\right\} p(S_{1}).
\]
\end{thm}

\begin{proof}
After applying localization, we have two cases, based on which of
the distance conditions in the theorem holds for a needle. If it is
the first case, the theorem follows from Theorem \ref{thm:Gaussian-iso-2}.
For the second case, it follows from Lemma 3.8 in \cite{KLS97}. 
\end{proof}
The next lemma is Lemma 6.4 from \cite{KannanLM06}. 
\begin{lem}[\cite{KannanLM06}]
\label{lem:coupling}Let $x,y\in K$ with $\norm{x-y}\le\delta/\sqrt{n}$.
Then 
\[
d_{TV}(P_{x,}P_{y})<1-d_{\ell}(x,y).
\]
\end{lem}

We are now ready to bound the log-Cheeger constant. 
\begin{proof}[Proof of Theorem \ref{thm:speedy-log-Cheeger} ]
 Note that the distribution with density proportional to $\ell(y)$
is nearly isotropic using Lemma \ref{lem:local}. Let $S\subset K$
have measure $s$ according to $\ell(y)$. Following a standard argument
partition $K$ as follows:

\begin{align*}
S_{1} & =\left\{ x\in S\,:\,P_{x}(K\setminus S)\le0.01\right\} \\
S_{2} & =\left\{ x\in K\setminus S\,:\,P_{x}(S)\le0.01\right\} \\
S_{3} & =K\setminus S_{1}\setminus S_{2}.
\end{align*}
Then by Lemma \ref{lem:coupling}, for every $x\in S_{1},y\in S_{2}$,
either $d_{\ell}(x,y)\ge d_{\ell}=\frac{1}{400}$ or $\norm{x-y}\ge d=\delta/\sqrt{n}$. 

To bound $p(S_{3})$, we consider the localization process $p_{t}$.
Without loss of generality, $p(S_{1})\leq p(S_{2})$. Since $p_{t}$
is a martingale, we have that
\begin{align*}
p(S_{3}) & =\E p_{T}(S_{3})\\
 & \gtrsim\E\left[\min\left\{ \frac{d_{\ell}}{\sqrt{n}},d\sqrt{T\log\frac{1}{\min\left(p_{T}(S_{1}),p_{T}(S_{2})\right)}}\right\} \min\left(p_{T}(S_{1}),p_{T}(S_{2})\right)\right]\\
 & \gtrsim\frac{1}{\sqrt{n}}\E\left[\min\left\{ 1,\delta\sqrt{T\log\frac{1}{p_{T}(S_{1})}}\right\} p_{T}(S_{1})1_{p_{T}(S_{1})\leq\frac{1}{2}}\right]
\end{align*}
where we used Theorem \ref{thm:gaussian-speedy} in the first inequality
and we used $d_{\ell}=\frac{1}{400}$ and $d=\delta/\sqrt{n}$.

Let $E$ be the event $\delta\sqrt{T\log\frac{1}{p_{T}(S_{1})}}\geq1$.
Under this event, we have that $\log\frac{1}{p_{T}(S_{1})}\geq\frac{1}{T\delta^{2}}$
and hence
\[
\delta\sqrt{T\log\frac{1}{p_{T}(S_{1})}}p_{T}(S_{1})\leq e^{-\frac{1}{T\delta^{2}}}.
\]
Hence, we have that
\[
p(S_{3})\gtrsim\delta\sqrt{\frac{T}{n}}\E\left[p_{T}(S_{1})\sqrt{\log\frac{1}{p_{T}(S_{1})}}1_{p_{T}(S_{1})\leq\frac{1}{2}}\right]-\frac{1}{\sqrt{n}}e^{-\frac{1}{T\delta^{2}}}.
\]
Take $T=\max\left(\frac{1}{D^{2}}\log\frac{1}{p(S_{1})},\frac{1}{\log\frac{1}{p(S_{1})}+D}\right)$,
Lemma \ref{lem:g_sqrt_g} shows that 
\[
p(S_{3})\gtrsim\delta\sqrt{\frac{T}{n}}p(S_{1})\sqrt{\log\frac{1}{p(S_{1})}}-\frac{1}{\sqrt{n}}e^{-\frac{1}{T\delta^{2}}}.
\]
Using $\delta\lesssim1$, for $\log\frac{1}{p_{T}(S_{1})}\gtrsim\frac{D}{\delta^{2}}$,
the equation is simply
\begin{align*}
p(S_{3}) & \gtrsim\delta\sqrt{\frac{T}{n}}p(S_{1})\sqrt{\log\frac{1}{p(S_{1})}}\\
 & \gtrsim\frac{\delta}{\sqrt{n}}p(S_{1})\left(\sqrt{\frac{\log\frac{1}{p(S_{1})}}{D}}+\frac{\log\frac{1}{p(S_{1})}}{D}\right)
\end{align*}

Going back to the conductance, for $\log\frac{1}{p(S)}\gtrsim\frac{D}{\delta^{2}}$,
we have
\begin{align*}
\frac{1}{2}\left(\int_{S}P_{x}(K\backslash S)d\ell(x)+\int_{K\backslash S}P_{x}(S)d\ell(s)\right) & \geq\frac{1}{2}\int_{S_{3}}(0.01)d\ell\\
 & \gtrsim\frac{\delta}{\sqrt{n}}p(S)\left(\sqrt{\frac{\log\frac{1}{p(S)}}{D}}+\frac{\log\frac{1}{p(S)}}{D}\right)
\end{align*}
Therefore, we have that
\[
\psi(s)\gtrsim\frac{\delta}{\sqrt{n}}\left(\sqrt{\frac{\log\frac{1}{s}}{D}}+\frac{\log\frac{1}{s}}{D}\right)\text{ if }s\geq e^{-\frac{cD}{\delta^{2}}}.
\]
\end{proof}

\section{\label{sec:tight}Optimality of the bounds}
\begin{lem}
\label{lem:lower_bound_logsob}For any $\frac{n}{2}\geq D\geq2\sqrt{n}$,
there exists an isotropic logconcave distribution with diameter $D$
such that its log-Cheeger constant is $O(1/\sqrt{D})$ and its log-Sobolev
constant is $O(1/D)$. 
\end{lem}

In fact, we get a nearly lower tight bound on the mixing time of the
ball walk, from an arbitrary start, in terms of the number of proper
steps. Recall that by a proper step we mean steps where the current
point changes, not counting the steps that are discarded due to the
rejection probability. In our lower bound example, the local conductance,
i.e., the probability of a proper step, is at least a constant everywhere
and so the total number of steps in expectation is within a constant
factor of the number of proper steps.
\begin{lem}
\label{lem:lower_bound_ball}For any $\frac{n}{2}\geq D\geq2\sqrt{n}$,
there exists an isotropic convex body with diameter $D$ such that
ball walk mixes in $\widetilde{\Omega}(n^{2}D)$ proper steps.
\end{lem}

Both theorems are based on the following cone
\[
K=\left\{ x:\:0\le x_{1}\le n,\sum_{i=2}^{n}x_{i}^{2}\le\frac{1}{n}x_{1}^{2}\right\} ,
\]
and the truncated cone
\[
K_{D}=K\cap\left\{ x:\ (x_{1}-n)^{2}+\sum_{i=2}^{n}x_{i}^{2}\leq D^{2}\right\} .
\]

\begin{proof}[Proof of Lemma \ref{lem:lower_bound_logsob}.]
The convex body $K_{D}$ is nearly isotropic and has diameter $D$.
Let
\[
t_{0}=n-\sqrt{D^{2}-n}.
\]
Consider the subset $S=K\cap\left\{ t_{0}\le x_{1}\le t_{0}+1\right\} $.
Note that $S$ is fully contained in $K_{D}$ and that
\[
p(S)=\frac{\vol(S)}{\vol(K_{D})}\lesssim\frac{\vol(S)}{\vol(K)}=\left(\frac{t_{0}+1}{n}\right)^{n-1}-\left(\frac{t_{0}}{n}\right)^{n-1}\leq e^{-\Omega(D)}.
\]
On the other hand, the expansion of $S$ is at most $2$. Therefore,
the log-Cheeger constant $\kappa$ of $K_{D}$ must satisfy 
\[
\kappa\sqrt{\log\frac{1}{p(S)}}\le2
\]
or $\kappa=O\left(D^{-\frac{1}{2}}\right)$ as claimed. It is known
that the log-Sobolev constant $\rho=\Theta(\kappa^{2})$ (see e.g.,
\cite{ledoux1994simple}). This gives the second claim.
\end{proof}
The proof of the lower bound for the ball walk is more involved. We
start the process with the uniform distribution on the set $S=K_{D}\cap\left\{ t_{0}\le x_{1}\le t_{0}+1\right\} $.
The distribution at any time will remain spherically symmetric at
each time step. Each step along the $e_{1}$ direction is approximately
$\pm\frac{\delta}{\sqrt{n}}=\pm\frac{1}{n}$. An unbiased process
that moves $\pm\frac{1}{n}$ along $e_{1}$ in each step would take
$\Omega(n^{2}D^{2})$ steps to converge since the diameter is effectively
$nD$. But there is a slight drift in the positive direction towards
the base. This is because for points near the boundary of the cone
a step away from the base has higher rejection probability compared
to a step towards the base. In the proof we will bound the drift from
above and therefore bound the number of steps needed from below.
\begin{proof}[Proof of Lemma \ref{lem:lower_bound_ball}.]
Let the starting distribution be uniform over $S=K_{D}\cap\left\{ t_{0}\le x_{1}\le t_{0}+1\right\} $.
For each point $x$, the local conductance is $\ell(x)=\frac{\vol(K_{D}\cap(x+\delta B))}{\vol(\delta B)}$,
the fraction of the $\delta$-ball around $x$ contained in $K_{D}$.
Let $\delta=\frac{1}{\sqrt{n}}$ and $D<0.99n$. 

On any slice $A(t)$, we can identify the bias felt by each point
$y\in A(t)$. It is the fraction of the $\delta$-ball around $y$
that is contained in $K$ but its mirror image through the plane $x_{1}=y_{1}$
is not contained in $K$. This fraction depends on the distance of
$y$ to the boundary of $A(t)$. Namely, the bias is $O(e^{-r^{2}/2}\delta)=O(e^{-r^{2}/2}/\sqrt{n})$
if the distance of $y$ to the boundary of $A(t)$ is $r\delta/\sqrt{n}=r/n$. 

Next, we observe that the density within each slice is spherically
symmetric. In fact, a stronger property holds. 
\begin{claim}
The distribution at any time in any cross-section is spherically symmetric
and unimodal. 
\end{claim}

Suppose the claim is true. Then, since the radius of any slice for
$t\ge t_{0}=n-\sqrt{D^{2}-n}$ is $\Omega(\sqrt{n})$, the fraction
of points at distance $r/n$ from the boundary is at most $1-(1-O(r/n^{1.5}))^{n}=O(r/\sqrt{n})$
fraction of the slice. Each point at distance $r/n$ has probability
$O(e^{-r^{2}/2}/\sqrt{n})$ of making a move towards the base with
no symmetric move away. Thus, effectively, each step of the process
is balanced along $e_{1}$ with probability $1-\Omega\left(\frac{e^{-r^{2}/2}}{\sqrt{n}}\right)$
and is $O\left(\frac{1}{n}\right)$ biased towards the base with probability
$O\left(\int_{r}\frac{e^{-r^{2}/2}}{\sqrt{n}}\cdot\frac{r}{\sqrt{n}}\,dr\right)=O\left(\frac{1}{n}\right)$.
Note that by the claim this holds for the density at every time and
on every slice. Therefore if we write $X(t)$ as the expected $x_{1}$
coordinated of the distribution after $t$ steps, we have 
\[
X(t+1)\le X(t)+O\left(\frac{1}{n^{2}}\right).
\]
In other words, the drift is $O\left(\frac{1}{n^{2}}\right)$ per
step and it takes $\Omega\left(n^{2}D\right)$ steps to cover the
$\Omega(D)$ distance to get within distance $1$ of the base along
$e_{1}$. This is necessary to have total variation distance less
than (say) $0.1$ from the target speedy distribution over the body.
\begin{figure}
\begin{centering}
\begin{tikzpicture}[y=0.80pt,x=0.80pt,yscale=-1, inner sep=0pt, outer sep=0pt]   \path[draw=black,line width=0.5pt] (60,124) -- (172.4,124);   \path[draw=black,line width=0.5pt] (60,124) -- (172.4,88);   \path[draw=black,line width=0.5pt] (60,124) -- (172.4,160);   \path[draw=black,line width=0.5pt] (173,124) ellipse (12 and 36);   \path[draw=black!25!green,line width=0.5pt] (72.9019,120) .. controls     (74.1409,121.6282) and (75.2116,123.9826) .. (72.9019,127.7602) node[pos=-0.5, above] {$\frac{1}{\sqrt{n}}$};        \path[draw=black!25!red,line width=0.5pt] (102.5654,111.2262) ellipse (0.1553cm and 0.1533cm);   \path[draw=black!25!blue,line width=0.5pt] (41.6054,167.1767) ellipse (0.6850cm and 0.6651cm);   \path[draw=black,dash pattern=on 0.68pt off 0.23pt,line width=0.5pt]     (41.7436,167.2131) ellipse (0.5588cm and 0.5332cm);   \path[draw=black!25!red,line width=0.5pt] (41.2551,146.2606) ellipse (0.1553cm and 0.1533cm);   \path[draw=black,line width=0.5pt] (67.8245,195.0162) --    (190.4452,166.4695);   \path[draw=black!25!red,line width=0.5pt] (133.0844,185.9210) ellipse     (0.6850cm and 0.6671cm);   \path[draw=black!25!blue,line width=0.5pt] (134.0008,179.4453) -- (133.3520,209.2895);   \path[draw=black,line width=0.5pt]  (134.0008,180.0941) --   (156.3839,192.0966);   \path[draw=black,line width=0.5pt] (139.1911,181.0672) -- (141.7862,178.7965);   \path[draw=black,line width=0.5pt] (143.4082,183.0136) -- (148.2741,178.1477);   \path[draw=black,line width=0.5pt] (146.9765,185.6087) -- (153.4644,179.1209);   \path[draw=black,line width=0.5pt] (151.8424,186.9063) -- (155.4107,183.0136);   \path[->, draw=black,dash pattern=on 2pt off 2pt,line width=0.5pt]     (102,110.5) --  (123.9446,159.0085);   \path[->,draw=black,dash pattern=on 2pt off 2pt,line width=0.5pt]     (98,128) --  (53,160);   \path[draw=black!25!blue,line width=0.5pt] (102,124) ellipse (4.5 and 13.5);
\end{tikzpicture}
\par\end{centering}
\centering{}\caption{The lower bound construction}
\end{figure}

We now prove the claim. We need to argue that starting at the distribution
at time $t$ in each slice is unimodal, i.e., radially monotonic nonincreasing
going out from the center. Consider two points $x,y$ on the same
slice with $x$ closer to the boundary. Without loss of generality,
we can assume they are on the same radial line from the center. Also,
$\ell(x)\le\ell(y)$. Let $B(x)=K_{D}\cap(x+\delta B)$. Every point
$z\in B(y)\setminus B(x)$ has $\ell(z)\ge\ell(w)$ for any point
in $w\in B(x)\setminus B(y)$. Suppose the current distribution is
$p_{t}$, which is a radially monotonic nonincreasing function. Then,
\begin{align*}
p_{t+1}(y)-p_{t+1}(x) & =\int_{z\in B(y)}\frac{p_{t}(z)}{\vol(\delta B)}dz+(1-\ell(y))p_{t}(y)-\int_{z\in B(x)}\frac{p_{t}(z)}{\vol(\delta B)}dz-(1-\ell(x))p_{t}(x)\\
= & \int_{z\in B(y)\backslash B(x)}\frac{p_{t}(z)}{\vol(\delta B)}dz-\int_{w\in B(x)\backslash B(y)}\frac{p_{t}(w)}{\vol(\delta B)}dz\\
 & +(1-\ell(y))p_{t}(y)-(1-\ell(x))p_{t}(x)
\end{align*}
The above expression is minimized by choosing $p_{t}$ to be as uniform
as possible (under the constraint of monotonicity, i.e., $p_{t}(y)\ge p_{t}(x)$
and $\forall z\in B(y)\setminus B(x),w\in B(x)\setminus B(y),\,p_{t}(z)\ge p_{t}(w)$)
in particular we can set $p_{t}(z)=p_{t}$ to be constant in $B(x)\cup B(y)$.
Then,
\begin{align*}
p_{t+1}(y)-p_{t+1}(x) & \ge p_{t}\ell(y)+(1-\ell(y))p_{t}-p_{t}\ell(x)-(1-\ell(x))p_{t}\\
 & =0.
\end{align*}
Thus, the new density maintains monotonicity as claimed. 
\end{proof}

\begin{acknowledgement*}
This work was supported in part by NSF Awards CCF-1563838, CCF-1717349,
CCF-1740551, CCF-1749609 and DMS-1839116. We thank Ravi Kannan, Laci
Lov\'{a}sz, Assaf Naor and Nisheeth Vishnoi for their support and encouragement.
We also thank S\'{e}bastien Bubeck, Ben Cousins, Ronen Eldan, Bo'az Klartag,
Anup B. Rao and Ting-Kam Leonard Wong for helpful discussions, and
Ronen for his wonderful invention of stochastic localization. Finally,
we thank Emanuel Milman for showing us the proof for small ball probability.
\end{acknowledgement*}

\appendix

\section{Bounding the spectral norm without anisotropic KLS}

Recall that Lemma \ref{lem:norm_At} shows that $\P(\max_{t\in[0,T]}\norm{A_{t}}_{\spe}\geq2)\leq2\exp(-\frac{1}{cT})$
for $0\leq T\leq\frac{1}{c\cdot kn^{1/k}}$ if the anisotropic KLS
$\psi_{p}\lesssim\left(\tr A^{k}\right)^{1/2k}$ holds. Since the
bound on $\|A_{t}\|_{\op}$ implies a bound on the isotropic KLS.
Therefore, it is interesting to get rid of the anisotropic condition,
namely, can we use the isotropic KLS to ``prove'' isotropic via
localization. The original Eldan paper shows this recursion gains
a $\sqrt{\log n}$ factor. In this section, we show the loss is $(\log n)^{\frac{1}{2}-\frac{1}{2k}}=(\log n)^{\frac{1}{4}}$
when $k=2$. We note that if each recursion losses a factor when $k=2$,
this implies a better KLS bound.

Now, we note that the anisotropic KLS assumption is only used in Lemma
\ref{lem:tensor_bound}.
\begin{lem}
\label{lem:tensor_bound2}Assume that for some $k\geq2$, $\psi_{n}\lesssim n^{1/2k}$
for all $n$. Given a logconcave distribution $p$ with mean $\mu$
and covariance $A$. For any $B^{(1)},B^{(2)},B^{(3)}\succeq0$, 
\begin{align*}
 & \left|\E_{x,y\sim p}(x-\mu)^{T}B^{(1)}(y-\mu)\cdot(x-\mu)^{T}B^{(2)}(y-\mu)\cdot(x-\mu)^{T}B^{(3)}(y-\mu)\right|\\
\lesssim & (\log n)^{1-\frac{1}{k}}\cdot\tr(A^{\frac{1}{2}}B^{(1)}A^{\frac{1}{2}})\cdot\left(\tr(A^{\frac{1}{2}}B^{(2)}A^{\frac{1}{2}})^{k}\right)^{1/k}\cdot\norm{A^{\frac{1}{2}}B^{(3)}A^{\frac{1}{2}}}_{\spe}.
\end{align*}
\end{lem}

\begin{proof}
Without loss of generality, we can assume $p$ is isotropic. Furthermore,
we can assume $B^{(1)}$ is diagonal. Let $\Delta^{(k)}=\E_{x\sim p}xx^{T}\cdot x^{T}e_{k}$.
Similar to the proof in Lemma \ref{lem:tensor_bound}, we have that
\begin{equation}
\E_{x,y\sim p}x^{T}B^{(1)}y\cdot x^{T}B^{(2)}y\cdot x^{T}B^{(3)}y\lesssim\tr B^{(1)}\cdot\tr(\Delta^{(k)}B^{(2)}\Delta^{(k)})\cdot\norm{B^{(3)}}_{\spe}.\label{eq:tensor_bound_2_key}
\end{equation}
and that
\[
\tr(\Delta^{(k)}B^{(2)}\Delta^{(k)})\leq\sqrt{\Var_{x\sim p}x^{T}B^{(2)}\Delta^{(k)}x}.
\]

Write $B^{(2)}=\sum_{i=0}^{\left\lceil \log n\right\rceil }C_{i}$
where each $C_{i}$ has eigenvalues in $(\|B^{(2)}\|_{\op}2^{i}/n,\|B^{(2)}\|_{\op}2^{i+1}/n]$
for $i\geq1$ and $C_{0}$ has eigenvalues less than $\|B^{(2)}\|_{\op}/n$.
Let $P_{i}$ be the orthogonal projection from $\Rn$ to the range
of $C_{i}$.
\begin{align*}
\Var_{x\sim p}x^{T}B^{(2)}\Delta^{(k)}x & =\sum_{i=0}^{\left\lceil \log n\right\rceil }\Var_{x\sim p}x^{T}P_{i}B^{(2)}\Delta^{(k)}x\\
 & \lesssim\sum_{i=0}^{\left\lceil \log n\right\rceil }\psi_{r_{i}}^{2}\cdot\E_{x\sim p}\|P_{i}B^{(2)}\Delta^{(k)}x\|_{2}^{2}\\
 & =\sum_{i=0}^{\left\lceil \log n\right\rceil }\psi_{r_{i}}^{2}\cdot\tr(\Delta^{(k)}C_{i}^{2}\Delta^{(k)})
\end{align*}
where $r_{i}=\rank(P_{i}B^{(2)}\Delta^{(k)}+(P_{i}B^{(2)}\Delta^{(k)})^{\top})\leq2\cdot\rank(P_{i})$.
Hence, we have that
\begin{align*}
\tr(\Delta^{(k)}B^{(2)}\Delta^{(k)})^{2} & \lesssim\sum_{i=0}^{\left\lceil \log n\right\rceil }\rank(P_{i})^{\frac{1}{k}}\cdot\|C_{i}\|_{\op}\cdot\tr(\Delta^{(k)}B^{(2)}\Delta^{(k)})\\
 & \lesssim\left(\|B^{(2)}\|_{\op}+\sum_{i=1}^{\left\lceil \log n\right\rceil }\rank(P_{i})^{\frac{1}{k}}\cdot\|C_{i}\|_{\op}\right)\cdot\tr(\Delta^{(k)}B^{(2)}\Delta^{(k)})
\end{align*}
where we used that $\rank(P_{0})^{\frac{1}{k}}\|C_{0}\|_{\op}\leq n\cdot\|B^{(2)}\|_{\op}/n\leq\|B^{(2)}\|_{\op}$
at the end. Finally, we note that 
\[
\sum_{i=1}^{\left\lceil \log n\right\rceil }\rank(P_{i})^{\frac{1}{k}}\|C_{i}\|_{\op}\lesssim\left(\sum_{i=1}^{\left\lceil \log n\right\rceil }\rank(P_{i})\|C_{i}\|_{\op}^{k}\right)^{\frac{1}{k}}(\log n)^{1-\frac{1}{k}}\leq\left(\tr(B^{(2)})^{k}\right)^{\frac{1}{k}}(\log n)^{1-\frac{1}{k}}
\]
and hence
\[
\tr(\Delta^{(k)}B^{(2)}\Delta^{(k)})^{2}\lesssim(\log n)^{1-\frac{1}{k}}\left(\tr(B^{(2)})^{k}\right)^{\frac{1}{k}}\tr(\Delta^{(k)}B^{(2)}\Delta^{(k)}).
\]
Hence, we have that
\[
\tr(\Delta^{(k)}B^{(2)}\Delta^{(k)})\leq(\log n)^{1-\frac{1}{k}}\cdot\left(\tr(B^{(2)})^{k}\right)^{1/k}.
\]
Putting this into (\ref{eq:tensor_bound_2_key}) gives the result.
\end{proof}
Using Lemma \ref{lem:tensor_bound2} instead of Lemma \ref{lem:tensor_bound}
in Section \ref{subsec:Bounding-the-spectral} gives the following
result:
\begin{lem}
\label{lem:norm_At2}Assume that for some $k\geq2$, $\psi_{n}\lesssim n^{1/2k}$
for all $n$. There is some universal constant $c\geq0$ such that
for any 
\[
0\leq T\leq\frac{1}{c\cdot k\cdot(\log n)^{1-\frac{1}{k}}\cdot n^{1/k}},
\]
we have that
\[
\P(\max_{t\in[0,T]}\norm{A_{t}}_{\spe}\geq2)\leq2\exp(-\frac{1}{cT}).
\]
\end{lem}

\bibliographystyle{plain}
\bibliography{acg}

\end{document}